\documentclass[]{amsart} \usepackage{amsfonts}

\usepackage{bbding}
\usepackage[dvips]{graphicx}
\usepackage{amsfonts,amssymb,amsmath,amsthm,amscd}

\usepackage{mathrsfs}
\usepackage{xcolor}
\usepackage{empheq}
\usepackage{adforn}
\usepackage{cancel}
\usepackage{xr}
\usepackage{mdframed}
\usepackage{tikz}
\usepackage{amsfonts}
\usepackage{mathdots}
\usepackage{pinlabel}

\usepackage{enumerate}
\usepackage{lmodern}
\usepackage{mdframed}
\usepackage{stmaryrd}
\usepackage{blindtext}
\usepackage{leftidx}

\usepackage{accents}

\definecolor{slightblue}{rgb}{.8, .8, 1}
\definecolor{tif}{RGB}{10, 186, 181}
\definecolor{hair}{RGB}{100,225,190}
\definecolor{ruby}{RGB}{220,50,120}
\definecolor{grass}{RGB}{150,220,110}

\usepackage[colorlinks=true, pdfstartview=FitP, linkcolor=tif!85!black, citecolor=ruby, urlcolor=blue!80!black]{hyperref}

\usepackage{tgpagella}
\usepackage[T1]{fontenc}

\makeatletter
\newtheorem*{rep@theorem}{\rep@title}
\newcommand{\newreptheorem}[2]{%
	\newenvironment{rep#1}[1]{%
		\def\rep@title{#2 \ref{##1}}%
		\begin{rep@theorem}}%
		{\end{rep@theorem}}}
\makeatother

\makeatletter
\newtheorem*{rep@cor}{\rep@title}
\newcommand{\newrepcor}[2]{%
	\newenvironment{rep#1}[1]{%
		\def\rep@title{#2 \ref{##1}}%
		\begin{rep@cor}}%
		{\end{rep@cor}}}
\makeatother

\makeatletter
\newtheorem*{rep@prop}{\rep@title}
\newcommand{\newrepprop}[2]{%
	\newenvironment{rep#1}[1]{%
		\def\rep@title{#2 \ref{##1}}%
		\begin{rep@prop}}%
		{\end{rep@prop}}}
\makeatother
\newtheorem{corollary}{Corollary}[section]
\newtheorem{corx}{Corollary}

\newtheorem{theorem}[corollary]{Theorem}
\newtheorem{thmx}[corx]{Theorem}

\newtheorem{proposition}[corollary]{Proposition}

\newrepcor{cor}{Corollary}
\newreptheorem{theorem}{Theorem}
\newrepprop{prop}{Proposition}

\newtheorem*{theorem*}{Theorem}
\newtheorem{lemma}[corollary]{Lemma}

\theoremstyle{definition} 
\newtheorem{definition}[corollary]{Definition}

\theoremstyle{remark} \newtheorem{remark}[corollary]{Remark} \numberwithin{equation}{section}
\newtheorem*{remark*}{Remark}
\numberwithin{figure}{section}

\newcommand{\R}{\mathbb{R}}

\newcommand{\Hom}{\mathrm{Hom}}

\renewcommand{\P}{\mathbb{P}}
\newcommand{\C}{\mathfrak{C}}
\newcommand{\hC}{\widehat{\mathfrak{C}}}

\newcommand{\ttri}{\widetilde{\mathcal{X}}}
\newcommand{\tri}{\mathcal{X}}

\newcommand{\ve}[1]{{\boldsymbol{#1}}}
\newcommand{\T}{\mathsf{T}}
\newcommand{\dif}{\mathsf{d}}

\newcommand{\ac}[1]{{\boldsymbol{#1}}} %complex structure
\newcommand{\transp}[1]{\leftidx{^\mathsf{t}}{#1}}
\newcommand{\para}{\ve{T}}

\newcommand{\pa}{\partial}
\newcommand{\bpa}{{\bar{\partial}}}

\newcommand{\dz}{{\dif z}}
\newcommand{\dbz}{{\dif\bar z}}

\newcommand{\ima}{\boldsymbol{i}} %imaginary unit

\newcommand{\GL}{\mathrm{GL}}
\newcommand{\SL}{\mathrm{SL}}

\newcommand{\SO}{SO}

\newcommand{\eg}{\textit{e.g. }}
\newcommand{\cf}{\textit{c.f. }}
\newcommand{\ie}{\textit{i.e. }}

\newcommand{\RP}{\mathbb{RP}^2}

\begin{document}

\title{Limit polygons of convex domains in the projective plane} 
\author{Xin Nie} 
\address{Tsinghua University, Beijing.}
\email{nie.hsin@gmail.com}
\maketitle

\begin{abstract}
For any sequence of properly convex domains in the real projective plane such that the zeros of Pick differentials have bounded multiplicity and get further and further apart, we determined all Hausdorff limit domains that one can obtain after normalizing each member of the sequence by a projective transformation. We then show that the result can be applied to convex domains generated by projective triangular reflection groups.
\end{abstract}

\section{Introduction}
A \emph{properly convex domain} in the real projective plane $\RP$ is an open set contained in some affine chart as a bounded convex domain. This paper follows the approach, initiated by Labourie \cite{labourie_cubic} and Loftin \cite{loftin_amer}, of studying such domains using \emph{hyperbolic affine spheres}. The approach is motivated by convex $\RP$-structures on surfaces  \cite{goldman_gx} and is found connected to the theory of Higgs bundles \cite{hitchin, labourie_cubic}. More specifically, we study limiting behaviors of properly convex domains, which is related to degenerations of convex $\RP$-structures.

Given an oriented properly convex domain $\Omega$, we consider the \emph{Pick differential} of the unique complete hyperbolic affine sphere $\Sigma_\Omega\subset\R^3$ projecting to $\Omega$ (in an orientation preserving way), which is a holomorphic cubic differential on $\Sigma_\Omega$ with respect to a canonical conformal structure. We refer to its push-forward to $\Omega$ as the Pick differential of $\Omega$ and denote it by  $\ve{\psi}_\Omega$. 
%As the main result of the paper, given a sequence of properly convex domains, under the assumption that the distances between Pick zeros go to infinity and the multiplicities of zeros are bounded, we determine all the Hausdorff limits that we can obtain after normalizing each member of the sequence by a projective transformtion. 
Letting $\mathfrak{C}$ be the space of properly convex domains in $\RP$, endowed with the natural topology given by Hausdorff distance, we show:
\begin{thmx}\label{thm_intro2}
	Let $\Omega_1,\Omega_2,\cdots, \Omega\in\mathfrak{C}$ and assume that the Pick differential $\ve{\psi}_{\Omega_i}$ satisfies:
	\begin{itemize}
		\item on every $\Omega_i$, the flat metric underlying $\ve{\psi}_{\Omega_i}$ is complete;
		\item the infimum of distances between the zeros of $\ve{\psi}_{\Omega_i}$ with respect to the flat metric tends to $+\infty$ as $i\to\infty$;
		\item the multiplicities of the zeros of $\ve{\psi}_{\Omega_i}$ have an upper bound independent of $i$.
	\end{itemize}
	Then there exists a sequence $(a_i)$ in $\SL(3,\R)$ such that $a_i(\Omega_i)$ converges to $\Omega$ in $\C$ if and only if $\Omega$ is either a triangle or a regular $(n+3)$-gon, with $n\geq1$ satisfying the condition that every $\ve{\psi}_{\Omega_i}$, apart from at most finitely many exceptions, has a zero of multiplicity $n$. 
\end{thmx}
Here, by a \emph{regular $k$-gon}, we mean a convex domain in $\RP$ projectively equivalent to a regular $k$-gon in the Euclidean plane $\mathbb{E}^2\subset\RP$.

% whereas $[\Omega]$ denotes the point in the topological quotient $\Cq$ represented by $\Omega$. Note that an alternative formulation of the convergence $[\Omega_i]\to[\Omega]$ in $\left.\mathfrak{C}\middle/\SL(3,\R)\right.$ is that there exists a sequence $(a_i)$ of projective transformations such that $a_i(\Omega_i)\to \Omega$ in $\C$.

On a closed surface $S$ of genus $g\geq2$, the aforementioned works of Labourie and Loftin give a one-to-one correspondence between convex $\RP$-structures and cubic differentials, and our theorem can be applied to sequences of convex $\RP$-structures whose corresponding cubic differentials are given by scaling a fixed one. Some asymptotic properties of such sequences have already been studied \cite{loftin_flat}, whereas the recent work of Dumas and Wolf \cite{dumas-wolf_limit} treats more general families. Theorem \ref{thm_intro2} implies:
\begin{corx}\label{coro_intro1}
Let $\Sigma$ be a closed Riemann surface, $\ve{\psi}$ be a nontrivial holomorphic cubic differential on $\Sigma$ and denote
$N_{\ve{\psi}}:=\big\{n\in\mathbb{N}\ \big|\ \text{$\ve{\psi}$ has a zero of  multiplicity $n$}\big\}$.
Let $(\zeta_i)$ be a sequence of complex numbers tending to $\infty$ and $\Omega_i\in\C$ be a developing image of  the convex projective structure on $\Sigma$ corresponding to the scaled cubic differential $\zeta_i\ve{\psi}$. Then there exists a sequence $(a_i)$ in $\SL(3,\R)$ such that $a_i(\Omega_i)$ converges to $\Omega\in\C$ if and only if either $\Omega$ is a triangle or a regular $(n+3)$-gon with $n\in N_\ve{\psi}$.
\end{corx}
While the proof of the ``if'' part of Theorem \ref{thm_intro2}, given in \S \ref{sec_regular}, is based on showing that the Blaschke metrics converge, the idea of the proof for the ``only if'' part, given in \S \ref{sec_limitconvex}, is that on one hand, a lemma of Benoist and Hulin \cite{benoist-hulin} about continuous dependence of affine spheres on convex domains implies that the Blaschke metric $g_i$ of $\Omega_i$ converges in the pointed Gromov-Hausdorff sense to that of $\Omega$; on the other hand, the assumptions on the zeros of $\ve{\psi}_{\Omega_i}$ allow us to classify all possible pointed Gromov-Hausdorff limits of $(g_i)$. 

The second part of the paper, \S \ref{sec_tits}, deals with properly convex domains generated by projective deformations of hyperbolic triangular reflection groups,
\begin{figure}[h]
	\centering
	\includegraphics[width=12cm]{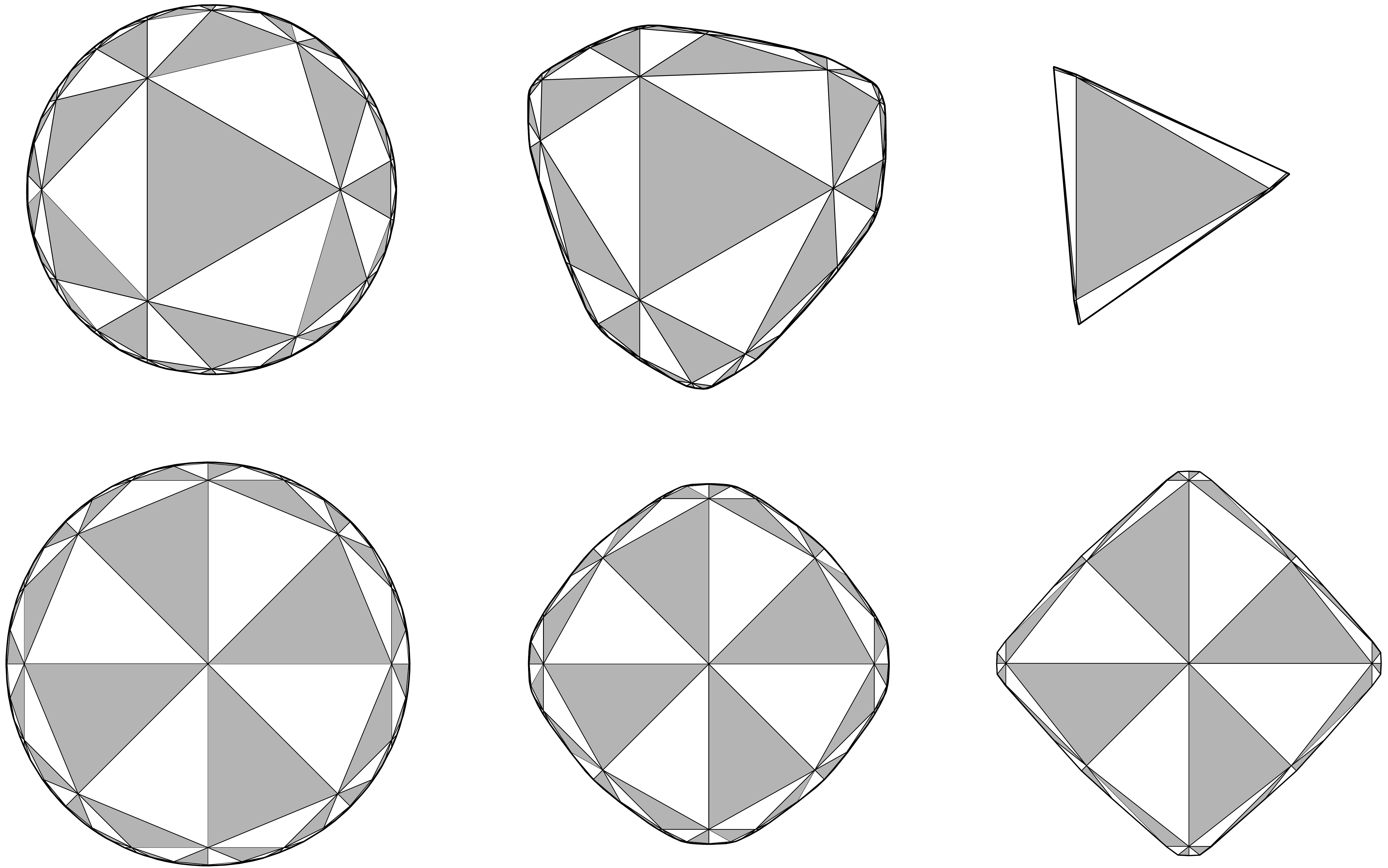}
\caption{Some properly convex domains generated by projective deformations of a hyperbolic $(4,4,4)$-triangle group. Each domain in the first row is projectively equivalent to the one underneath, although the two rows have different limit domains. 
	}
	\label{figure_tits}
\end{figure} 
which provide concrete examples for Theorem \ref{thm_intro2} and Corollary \ref{coro_intro1}. It has been observed from computer graphics (\cf \cite{nie_tits}) that when one goes far away in the moduli space of such deformations, the corresponding convex domain can be close to either a single fundamental triangle or a regular polygon formed by the fundamental triangles around a vertex, as illustrated in the two rows of Figure \ref{figure_tits} respectively. We deduce from Theorem \ref{thm_intro2} that these are the only possible limiting patterns:
\begin{thmx}\label{thm_intro3}
	Let $\Gamma_\Delta\subset\SO(2,1)\subset\SL(3,\R)$ be the hyperbolic reflection group generated by a triangle $\Delta\subset\mathbb{H}^2\subset\RP$ with angles $\frac{\pi}{a}$, $\frac{\pi}{a}$ and $\frac{\pi}{c}$ ($a,b,c\in\mathbb{N}$). Assume $a,b,c\geq 3$, so that the moduli space $\tri_\Delta$ of representations $\rho:\Gamma_\Delta\to\SL(3,\R)$ that are deformations of the inclusion is homeomorphic to $\R$ (see Proposition \ref{prop_para1}). Pick $\rho_t:\Gamma_\Delta\to\SL(3,\R)$ corresponding to $t\in\R\cong\tri$ for each $t$, and let $\Omega_t\in\mathfrak{C}$ be the properly convex domain preserved by $\rho_t(\Gamma_\Delta)$ (see Proposition \ref{prop_titsset}). Then there exist $(a_t)_{t\in\R}\subset\SL(3,\R)$ such that $a_t(\Omega_t)$ converges to $\Omega\in\C$ as $t$ goes to $+\infty$ or $-\infty$ if and only if $\Omega$ is a regular $k$-gon with $k\in\{3,a,b,c\}$. 
\end{thmx}
This result demonstrates the usefulness of cubic differentials as a tool in the study of convex $\RP$-structures, as the statement itself does not involve cubic differentials.
The reason why Theorem \ref{thm_intro2} applies is that the reflectional symmetries of $\Omega_t$ allow us to describe the Pick differential $\ve{\psi}_{\Omega_t}$, and in particular show that it gives a scaling family, and that a point $p\in\Omega_t$ is a zero of multiplicity $n$ if and only if $p$ is the common vertex of $n+3$ shaded triangles in Figure \ref{figure_tits}. See \S \ref{subsec_at} for details.

We finish this introduction by some remarks about the hypotheses of Theorem \ref{thm_intro2}. First, if the boundedness hypothesis on multiplicity is removed, we believe that the method here can be extended to prove a similar conclusion, with two additional types of limits besides regular polygons: ellipse and infinite-sided polygon whose vertices form an orbit of a parabolic projective transformation. In fact, both types arise as limits of sequences $(D_k)_{k=3,4,\cdots}$ in $\C$ where $D_k$ is a regular $k$-gon.
On the other hand, the hypothesis that the zeros of $\ve{\psi}_{\Omega_i}$ get further and further apart can be weakened to the condition that the zeros  form clusters that get further and further apart, and each cluster converges in certain sense to a polynomial cubic differential on $\mathbb{C}$. Under this condition, a similar conclusion also holds as long as one adds the non-regular polygons corresponding to those polynomial cubic differentials (in the sense of \cite{dumas-wolf}) to the limits. This situation is studied in \cite{dumas-wolf_limit}, the significance being that the condition is satisfied subsequentially for any sequence of cubic differentials on a sequence of closed Riemann surfaces that stay in a compact part of the Teichm\"uller space $\mathcal{T}_g$.
\subsection*{Acknowledgements}
We are grateful to Tengren Zhang for his interest in the subject, and to National University of Singapore and Yau Mathematical Sciences Center for the hospitality during the preparation of this paper.

%\setcounter{tocdepth}{1}
%\tableofcontents

\section{Preliminaries}\label{sec_preliminaries}
We give in this section a concise overview of the constructions relating affine spheres, cubic differentials and properly convex domains. Details and references can be found in \cite{benoist-hulin, dumas-wolf, loftin_neck, nie_poles}.
% Here we put some emphasis on the orientation-dependent nature of  Pick differentials, which is not much addressed in the literature but is involved in \S \ref{sec_tits} below.
\subsection{Pick differential}\label{subsec_affinesphere}
We first outline the construction of the Pick differential $\ve{\psi}_\Omega$ for a properly convex domain $\Omega\subset\RP$. As mentioned, its depends on a choice of orientation on $\Omega$. So we denote
$$
\hC:=\left\{\text{ oriented properly convex domains in $\RP$ } \right\},
$$ 
which is a two-fold covering of the space $\C$ of properly convex domains. Given $\Omega\in\hC$, let $-\Omega$ denote the same properly convex domain with opposite orientation.

An element $\Omega\in\hC$ can be interpreted alternatively as a domain on the projective sphere $\mathbb{S}^2:=(\R^3\setminus\{0\})/\mathbb{R}_+$, hence is the projection of a unique properly convex cone $C_\Omega$ in $\R^3$. This gives a one-to-one correspondence between elements of $\hC$ and such cones. With this interpretation, $-\Omega$ is just the domain in $\mathbb{S}^2$ opposite to $\Omega$. 

Given $\Omega\in\hC$, a theorem of Cheng and Yau ensures that the cone $C_\Omega$ contains a unique hyperbolic affine sphere $\Sigma_\Omega$ asymptotic to the boundary $\pa C_\Omega$, whereas works of Cheng-Yau, A. M. Li, C. P. Wang, Au-Wan, among others, imply that the affine-differential-geometric intrinsic invariants of $\Sigma_\Omega$ reduce to a pair $(\ac{J},\ve{\psi})$, where $\ac{J}$ is a complex structure on $\Sigma_\Omega$ (compatible with the orientation) and $\ve{\psi}$ a holomorphic cubic differential on the Riemann surface $(\Sigma_\Omega,\ac{J})$.
%In this paper, we often omit the discussion of the $\ac{J}$ part. Namely, when talking about a cubic differential on a surface without specifying the complex structure, we actually mean such a pair. 

The projection $\R^3\setminus\{0\}\to\RP$ induces an orientation-preserving diffeomorphism $\Sigma_\Omega\overset\sim\to\Omega$. The Pick differential $\ve{\psi}_\Omega$ is then defined as the push-forward of the Pick differential of $\Sigma_\Omega$. It is a holomorphic cubic differential with respect to the complex structure $\ac{J}_\Omega$ which is the push-forward of the above $\ac{J}$, and has the following fundamental properties:
\begin{itemize}
	\item $\ve{\psi}_\Omega$ is equivariant with respect to $\Omega\in\hC$ in the sense that $a_*\ve{\psi}_\Omega=\ve{\psi}_{a(\Omega)}$ for any $a\in\SL(3,\R)$;
	\item Reversing the orientation of $\Omega$ yields the conjugate cubic differential $\ve{\psi}_{-\Omega}=\overline{\ve{\psi}_\Omega}$, which is holomorphic with respect to the opposite complex structure $\ac{J}_{-\Omega}=-\ac{J}_\Omega$;
	\item The assignment $\Omega\mapsto (\Omega,\ve{\psi}_\Omega)$ gives a one-to-one correspondence between projective equivalence classes of oriented properly convex domains and conformal equivalence classes of pairs $(X,\ve{\psi})$, where $X$ is a contractible Riemann surface (conformally either  $\mathbb{C}$ or the unit disk $\mathbb{D}$) and $\ve{\psi}$ is a holomorphic cubic differential on $X$, with the constraint that if $X\cong\mathbb{C}$ then $\ve{\psi}$ is not constantly zero. 
	\item If $\Omega$ is an ellipse then $(\Omega,\ac{J}_\Omega)\cong\mathbb{D}$ and $\ve{\psi}_\Omega\equiv0$. In this case $C_\Omega$ is a quadratic cone and $\Sigma_\Omega$ is a component of a two-sheeted hyperboloid.
\end{itemize}

\subsection{Blaschke metric and Wang's equation}\label{subsec_blaschke}
For a hyperbolic affine sphere $\Sigma\subset\R^3$, an intrinsic invariant more fundamental than the Pick differential $\ve{\psi}$ is a Riemannian metric $g$ on $\Sigma$ called the \emph{Blaschke metric}, and the above complex structure $\ac{J}$ is just the one underlying $g$. However, in the case where $\Sigma$ is \emph{complete}, as considered above, the information of $g$ is fully covered by $\ve{\psi}$, because on one hand, $g$ and $\ve{\psi}$ satisfy \emph{Wang's equation}
\begin{equation}\label{eqn_wang}
\kappa_g=-1+2\|\ve{\psi}\|^2_g~,
\end{equation}
($\kappa_g$ is the curvature of $g$ and $\|\ve{\psi}\|_g$ the pointwise norm of $\ve{\psi}$ with respect to $g$); on the other hand, any Riemann surface endowed with a holomorphic cubic differential $\ve{\psi}$ carries a unique complete conformal metric $g$ satisfying the equation (excluding the case where the surface is not hyperbolic and $\ve{\psi}\equiv0$, see \cite[Thm. 1.1]{nie_poles}). 

We refer to any conformal metric $g$ on $\Sigma$ satisfying (\ref{eqn_wang}) as a \emph{solution} to Wang's equation for $(\Sigma,\ve{\psi})$. The curvature $\kappa_g$ of a solution can change sign in general, but the complete solution always has nonpositive curvature (\cite[Thm. 1.1]{nie_poles}).

Given $\Omega\in\hC$, similarly as the above definition of $\ve{\psi}_\Omega$, we define the Blaschke metric $g_\Omega$ of $\Omega$ as the push-forward of the Blaschke metric on the complete affine sphere $\Sigma_\Omega$. Alternatively, $g_\Omega$ is the unique complete solution to Wang's equation for $(\Omega,\ve{\psi}_\Omega)$. Note that $g_\Omega$ is still equivariant with respect to $\Omega$, but is orientation-independent in the sense that $g_{-\Omega}=g_\Omega$.

\subsection{Dual convex domain and affine sphere}\label{subsec_duality}
Let $\R^{3*}$ denote the vector space dual to $\R^3$, formed by all the linear forms on $\R^3$. The dual projective space $\mathbb{RP}^{2*}:=\P(\R^{3*})$ is the space of projective lines in $\RP$. Given $\Omega\in\C$, the \emph{dual domain} $\Omega^*$ is the properly convex domain in $\mathbb{RP}^{*2}$ formed by the projective lines disjoint from the closure $\overline{\Omega}$.

If $\Omega$ is endowed with an orientation, we can accordingly endow $\Omega^*$ with an orientation, such that the corresponding cones $C_\Omega\subset\R^3$ and $C_{\Omega^*}\subset\R^{*3}$ are dual to each other in the sense that $C_{\Omega*}$ is the set of linear forms on $\R^3$ taking positive values on $\overline{C}_\Omega\setminus\{0\}$. 

The complete hyperbolic affine spheres $\Sigma_\Omega$ and $\Sigma_{\Omega^*}$ in the cones $C_\Omega$ and $C_{\Omega^*}$ also turn out to be dual to each other in the following sense: Given a hyperbolic affine sphere $\Sigma\subset\R^3$, for each $v\in\Sigma$ we let $v^*\in\R^{*3}$ denote the linear form characterized by $v^*(v)=1$ and $v^*|_{\T_v\Sigma}=0$, where the tangent space $\T_v\Sigma$ is viewed as a subspace of $\R^3$ through translation. Then $\Sigma^*:=\{v^*\mid v\in\Sigma\}$ is a hyperbolic affine sphere in $\R^{*3}$, referred to as the affine sphere \emph{dual} to $\Sigma$.
\subsection{Wang's developing map}\label{subsec_wang}
A hyperbolic affine sphere can be reconstructed from its Blaschke metric and Pick differential through an integration procedure, called \emph{Wang's developing map} in the literature. In this subsection we review this constructed, formulated in the language of vector bundles and connections.

Given a contractible Riemann surface $X$ (conformally $\mathbb{C}$ or $\mathbb{D}$) and a holomorphic cubic differential $\ve{\psi}=\psi(z)\dz^3$ on $X$, assuming $\ve{\psi}\not\equiv0$ if $X\cong\mathbb{C}$, we let $g=e^{u(z)}|\dz|^2$ be the unique complete conformal metric satisfying Wang's equation \eqref{eqn_wang} as mentioned in \S \ref{subsec_blaschke}, and define the \emph{affine sphere connection} $\nabla_{\ve{\psi}}$ as the flat connection on the rank $3$ complex vector bundle 
$$
E_\mathbb{C}:=\T_\mathbb{C}X\oplus\underline{\mathbb{C}}
$$ 
over $X$ (where $\underline{\mathbb{C}}$ denotes the trivial line bundle over $X$ endowed with a canonical section $\underline{1}$) given by the following matrix expression under the frame $(\pa_z,\pa_{\bar z}, \underline{1})$:
$$
\nabla_{\ve{\psi}}=\dif+
\begin{pmatrix}
\pa u& e^{-u}\bar \psi\,\dbz&\dz\\[6pt]
e^{-u}\psi\,\dz&\bpa u&\dbz\\[6pt]
\frac{1}{2}e^u\dbz&\frac{1}{2}e^u\dz&0
\end{pmatrix}.
$$
Eq.\eqref{eqn_wang} implies that $\nabla_{\ve{\psi}}$ is flat and preserves the real sub-bundle
$$
E:=\T X\oplus\underline{\mathbb{R}}\subset E_\mathbb{C}.
$$ 
The flat vector bundle $(E,\nabla_{\ve{\psi}})$ is actually isomorphic to the one induced by an affine spherical immersion $f:X\to\R^3$ with Blaschke metric $g$ and Pick differential $\ve{\psi}$, hence we can recover $f$ from $\nabla_{\ve{\psi}}$ via the following definition and proposition.

\begin{definition}\label{def_wangdeveloping}
Given $X$, $\ve{\psi}$ as above and a base point $z_0\in X$, we let
$$
\para(z): E|_{z}\to E|_{z_0},\quad 
$$
denote the parallel transport of the flat connection $\nabla_\ve{\psi}$ along any path from $z\in X$ to $z_0$, and define \emph{Wang's developing map} associated to the data $(X,\ve{\psi},z_0)$ as
$$
f:X\to E|_{z_0}\cong\R^3,\ f(z):=\para(z)\underline{1}(z).
$$
\end{definition}
%Note that in the above discussions we do not required $g$ to be complete. As mentioned in \S \ref{subsec_blaschke}, if $g$ is complete then it is fully determined by $\ve{\psi}$, hence we refer to 
\begin{proposition}\label{prop_wangdeveloping}
The above map $f$ is a complete hyperbolic affine spherical immersion with Pick differential $\ve{\psi}$, and its dual affine spherical immersion (in the sense of \S \ref{subsec_duality}) is
$$
f^*:X\to E|_{z_0}^*\cong\R^{*3},\ f^*(z):=\para^*(z)\underline{1}^*(z),
$$
where
$\para^*(z)=\transp\para(z)^{-1}:E|_z^*\to E|_{z_0}^*$
is the parallel transport of the flat connection $\nabla_\ve{\psi}^*$ (the connection dual to $\nabla_\ve{\psi}$, defined on the dual vector bundle $E^*$) from $z$ to $z_0$. 
\end{proposition}
See \eg \cite[\S 4]{loftin_flat}, \cite[\S 3.4]{loftin_neck} for an equivalent formulation of the proposition. The main consequence of the proposition used latter on is that we can recover any $\Omega\in\hC$, up to projective equivalence, from $\ve{\psi}_\Omega$. More precisely, if $f$ is  Wang's developing map associated to $(\Omega,\ve{\psi}_\Omega,p)$ for any $p\in\Omega$, then the projectivized image $\P(f(\Omega))\subset\P(\T_p\Omega\oplus\mathbb{R})\cong\RP$ is projectively equivalent to $\Omega$.

As a specific property of the map $f$ in Definition \ref{def_wangdeveloping}, note that at the base point $z_0$, it takes the value $f(z_0)=\underline{1}_{z_0}$ and the differential $\dif f_{z_0}$ is just the inclusion of $\T_{z_0}X\hookrightarrow E|_{z_0}$. Therefore, the image $f(X)$ is contained in the affine half-space
$$
\left\{v+t\,\underline{1}_{z_0}\in E|_{z_0}=\T_{z_0}X\oplus\R\,\Big|\ v\in\T_{z_0}X,\ t\geq0\right\}
$$
because $f(X)$ is a complete convex surface and the origin $0\in E|_{z_0}$ lies on the concave side of $f(X)$. 

\subsection{Results on Wang's equation}\label{subsec_estimates}
 We need two more results about Wang's equation (\ref{eqn_wang}).  The first deals with \emph{supersolutions} and \emph{subsolutions} of the equation, defined respectively as conformal metrics $g_+$ and $g_-$ satisfying
 $$
 \kappa_{g_+}\geq -1+2\|\ve{\psi}\|^2_{g_+},\quad  \kappa_{g_-}\leq -1+2\|\ve{\psi}\|^2_{g_-}
 $$
(see \eg \cite[\S 2.1]{nie_poles} for details). In particular, given a solution $g$ to Wang's equation, since the curvature of the scaled metric $cg$ ($c>0$) is given by
 $$
 \kappa_{cg}=\tfrac{1}{c}\kappa_g=\tfrac{1}{c}\left(-1+2\|\ve{\psi}\|^2_g\right)=\tfrac{1}{c}\left(-1+2c^3\|\ve{\psi}\|^2_{cg}\right),
 $$
 $c g$ is a supersolution if $c\geq1$ and a subsolution if $c\leq1$. 
 In general, we have the following Comparison Principle (see \cite[Corollary 2.5]{nie_poles} for a proof):
 \begin{lemma}\label{lemma_max}
%	Let $\ve{\psi}=\psi(z)\dz^3$ be a holomorphic cubic differential on an open set $U\subset\mathbb{C}$ and $g_+$ (resp. $g_-$) be a smooth supersolution (resp. subsolution) to Wang's equation on $(U,\ve{\psi})$. Suppose we have $g_+\geq g_-$ on the boundary of a compact set $K\subset U$. Then $g_+\geq g_-$ holds throughout $K$.
Let $\Sigma$ be a Riemann surface, $\ve{\psi}$ be a holomorphic cubic differential on $\Sigma$ and $g_+$ (resp. $g_-$) be a supersolution (resp. subsolution) to Wang's equation for $(\Sigma,\ve{\psi})$, such that $g_+\geq g_-$ on the boundary of a compact set $K\subset \Sigma$. Then $g_+\geq g_-$ on the whole $K$.
\end{lemma}

%\begin{proof}
%	Let $g$ be a conformal Riemannian metric on $\tt{S}$ and assume $g_\pm=e^{u_\pm}g$. Since 
%	$$
%	\kappa_{e^ug}=e^{-u}\left(\kappa_g-\frac{1}{2}\Delta_gu\right),
%	$$
%	$g_\pm$ being super-/sub- solutions to Wang's equation means 
%	$$
%	\frac{1}{2}\Delta_gu_+\leq e^{u_+}-e^{-2u_+}\|B\|^2_g+\kappa_g,
%	$$ 
%	$$
%	\frac{1}{2}\Delta_gu_-\geq e^{u_-}-e^{-2u_-}\|B\|^2_g+\kappa_g.
%	$$
%	Suppose by contradiction that $u_+-u_-$ takes negative value somewhere in the interior $A^\circ$ of $A$, then by compactness it attends negative minimum at some $x_0\in A^\circ$. The above inequalities imply that 
%	$$
%	0\leq \frac{1}{2}\Delta_g(u_+-u_-)(x_0)\leq \left(e^{u_+}-e^{-2u_+}\|B\|^2_g\right)(x_0)-\left(e^{u_-}-e^{-2u_-}\|B\|^2_g\right)(x_0)<0,
%	$$
%	a contradiction. Here the last inequality is because the function $t\mapsto e^t-\|B\|^2_g(x_0)e^{-2t}$ is strictly increasing, whereas $u_+(x_0)<u_-(x_0)$.
%\end{proof}

The second result is a control of the curvature $\kappa_g$ at the center of a large flat disk when the curvature is nonpositive. Here  $B(0,R):=\{z\in\mathbb{C}\mid |z|<R\}$. 
\begin{lemma}\label{lemma_centerestimate}
There is a function $\lambda:\R_+\to\R_+$ with $\lim_{R\to+\infty}\lambda(R)=0$ such that for any nonpositively curved solution $g$ to Wang's equation for a Riemann surface $\Sigma$ with cubic differential $\ve{\psi}$, if $p\in \Sigma$ has a neighborhood on which $|\ve{\psi}|^\frac{2}{3}$ is isometric to the Euclidean disk $(B(0,R),\,|\dz|^2)$, with $p$ corresponding to $0$, then 
$$
e^{-\lambda(R)}\leq \kappa_g(p)+1=2\|\ve{\psi}\|^2_g(p)\leq 1.
$$
\end{lemma}
Although the result can be stated more simply as $\kappa_g(p)\geq-\mu(R)$ for a positive function $\mu$ with $\lim_{R\to+\infty}\mu(R)=0$, we use the above formulation because in the application in \S \ref{subsec_convergencewang}, we use the lemma to control the term $\|\ve{\psi}\|^2_g$. A proof of the lemma can be found in \cite[Thm. 2.15]{nie_poles}, where an explicit $\lambda$ with exponential decay is given. However, for the application in this paper, we do not need any information about the rate of the convergence $\lambda(R)\to0$.
\section{Regular polygons are limit domains}\label{sec_regular}
We prove in this section the ``if'' part of Theorem \ref{thm_intro2} by establishing a more precise result, Proposition \ref{prop_developingconvergence} below.
\subsection{Convergence of Wang's developing image}\label{subsec_convergencewang}
The following proposition can be viewed as specifying a normalization of the properly convex domains $\Omega_i$ in Theorem \ref{thm_intro2} which makes them converge. Here $B(0,r)$ denotes the disk $\{z\in\mathbb{C}\mid |z|<r\}$.
\begin{proposition}\label{prop_developingconvergence}
Let $n$ be a nonnegative integer. For $i=1,2,\cdots$, let
\begin{itemize}
	\item $X_i$ be a contractible Riemann surface;
	\item $\ve{\psi}_i$ be a nontrivial holomorphic cubic differential on $X_i$;
	\item $p_i\in X_i$ be a point which is a zero of $\ve{\psi}_i$ of multiplicity $n$ if $n\geq1$ and is an ordinary point of $\ve{\psi}_i$ if $n=0$;
    \item $U_i\subset X_i$ be a neighborhood of $p_i$, and $r_i$ be a positive number, tending to $\infty$ as $i\to\infty$, such that there is an isomorphism
	\begin{equation}\label{eqn_diskneighborhood}
	(U_i,\ve{\psi}_i)\cong(B(0,r_i),\, z^n\dz^3)
	\end{equation}
with $p_i$ corresponding to $0$;
\item $f_i:X_i\to \T_{p_i}X_i\oplus\R\cong\T_0\mathbb{C}\oplus\R$ be Wang's developing map associated to $(X_i,\ve{\psi}_i,p_i)$ (see  Definition \ref{def_wangdeveloping}), where we identify $\T_{p_i}X_i\cong \T_0\mathbb{C}$ through \eqref{eqn_diskneighborhood}.
\end{itemize}
Let $f:\mathbb{C}\to \T_0\mathbb{C}\oplus\R$ denote Wang's developing map associated to $(\mathbb{C},z^n\dz^3,0)$.
Then the projectivized image $\P(f_i(X))$, as a properly convex domain in $\P( \T_0\mathbb{C}\oplus\R)\cong\RP$, converges to $\P(f(\mathbb{C}))$ in the Hausdorff sense as $i\to\infty$.
\end{proposition}
We postpone the proof to the subsequent sections and proceed to give:
\begin{proof}[Proof of the ``if'' part of Theorem \ref{thm_intro2}]
Let $(\Omega_i)$ be as in the hypotheses of the theorem and $n$ be either $0$ or a positive integer such that every $\ve{\psi}_{\Omega_i}$ has a zero $p_i\in\Omega_i$ of multiplicity $n$. 
In the latter case, since the flat metric $h_i:=|\ve{\psi}_i|^\frac{2}{3}$ is complete and the $h_i$-distance from $p_i$ to the other zeros of $\ve{\psi}_{\Omega_i}$ tends to $\infty$, we can find a neighborhood $U_i$ of $p_i$ such that $(U_i,\ve{\psi}_{\Omega_i})\cong (B(0,r_i),z^n\dz^3)$, with $r_i\to\infty$, as required in the assumptions of Proposition \ref{subsec_convergencewang}. In the case $n=0$, we can also pick an ordinary point $p_i\in\Omega_i$ and a neighborhood $U_i$ of $p_i$ with the same property. Applying Proposition \ref{prop_developingconvergence}, we conclude that the projectivized image $\Omega_i'$ of Wang's developing map $f_i$ associated to $(\Omega_i,\ve{\psi}_{\Omega_i}, p_i)$ converges to $\P(f(\mathbb{C}))$. This implies the required statement because $\Omega_i'$ is projectively equivalent to $\Omega_i$ (see \S \ref{subsec_wang}) whereas $\P(f(\mathbb{C}))$ is a regular $(n+3)$-gon by \cite{dumas-wolf}.
\end{proof}

To prove Proposition \ref{prop_developingconvergence}, we first show in Proposition \ref{prop_metricconvergence} below that the complete solution $g_i$ of Wang's equation on $(X_i,\ve{\psi}_i)$ compactly $C^2$-converges to the solution on $(\mathbb{C},\,z^n\dz^3)$, which implies the same convergence property for the corresponding affine sphere connections. Then in \S \ref{subsec_poof} we deduce the required convergence of Wang's developing maps from the basic fact, reviewed in \S \ref{subsec_ode}, that the solution to a linear ODE system depends continuously on parameters.

\subsection{Convergence of solutions of Wang's equation}\label{subsec_convergencewang}
Given a nonnegative integer $n$, let $g_0^{(n)}$ denote the complete solution to Wang's equation for $(\mathbb{C},\,z^n\dz^3)$ (see \S \ref{subsec_blaschke}). The analytic results in \S \ref{subsec_estimates} allow us to show:
\begin{proposition}\label{prop_metricconvergence}
Let $n\geq0$ be an integer, $(r_i)_{i=1,2,\cdots}$ be a sequence of positive numbers tending to $\infty$, and $g_i$ be a solution to Wang's equation on $(B(0,r_i), z^n\dz^3)$ with nonpositive curvature. Then for any $r>0$, the sequence of metrics $(g_i)_{i\geq i_0}$  $C^2$-converges on $B(0,r)$ to $g_0^{(n)}$ (where $i_0$ is such that $r_i>r$ for all $i\geq i_0$).
\end{proposition}

\begin{proof}
We first establish $C^0$-convergence. Namely, given any $\epsilon>0$, we shall show
\begin{equation}\label{eqn_gg0}
e^{-\epsilon}g_0^{(n)}\leq g_i\leq e^\epsilon g_0^{(n)}
\end{equation}
in $B(0,r)$ when $i$ is big enough.

To this end, let $h:=|z|^\frac{2n}{3}|\dz|^2$ denote the flat metric on $\mathbb{C}$ underlying the cubic differential $z^n\dz^3$. The $h$-distance from any $z\in\mathbb{C}$ to $0$ can be written as $R(|z|)$, with
$$
R(t)=\left(\tfrac{n}{3}+1\right)^{-1}\,t^{\frac{n}{3}+1}.
$$
Let $B_h(z,R)$ denote the disk centered at $z$ with radius $R$ under the metric $h$.

Consider the function $\lambda$ from Lemma \ref{lemma_centerestimate}. Since $r_i\to+\infty$ and $\lim_{R\to\infty}\lambda(R)=0$, we can pick a big enough $i_0$ such that for all $i\geq i_0$ we have
$$
R(r_i)/2\geq R(r),\quad \lambda(R(r_i)/2)\leq\epsilon.
$$ 
Note that the first inequality means $B:=B_h(0,R(r_i)/2)$ contains $B(0,r)=B_h(0,R(r))$. 

For every $z\in\pa B$, we have 
$$
B_h(z,R(r_i)/2)\subset B_h(0,R(r_i))=B(0,r_i).
$$
The disk $B_h(z,R(r_i)/2)$, endowed with the metric $h$, is isometric to a Euclidean disk of radius $R(r_i)/2$ because it does no contain the singularity $0$ of $h$. So we can apply Lemma \ref{lemma_centerestimate}  to the metrics $g_i$ and $g_0^{(n)}$ on this disk, and conclude that 
$$
e^{-\lambda(R(r_i)/2)}\leq 2\|z^n\dz^3\|^2_{g_i}\leq 1,\quad e^{-\lambda(R(r_i)/2)}\leq 2\|z^n\dz^3\|^2_{g_0^{(n)}}\leq 1
$$
at every point $z$ on $\pa B$. Therefore, at such $z$, the conformal ratio
$$
\frac{g_i}{g_0^{(n)}}=\left(\frac{\|z^n\dz^3\|^2_{g_0^{(n)}}}{\|z^n\dz^3\|^2_{g_i}}\right)^\frac{1}{3}
$$
is bounded from above by $e^{\frac{1}{3}\lambda(R(r_i)/2)}\leq e^\frac{\epsilon}{3}$ and from below by $e^{-\frac{1}{3}\lambda(R(r_i)/2)}\geq e^{-\frac{\epsilon}{3}}$ for all $i\geq i_0$. This means the required inequality \eqref{eqn_gg0} holds on $\pa B$. But since $e^\epsilon g_0^{(n)}$ and $e^{-\epsilon} g_0^{(n)}$ are supersolution and subsolution to Wang's equation, respectively (see \S \ref{subsec_estimates}), we can apply Lemma \ref{lemma_max} to the pairs of metrics $(g_-,g_+)=(e^{-\epsilon}g_0^{(n)},g_i)$ and $(g_i,e^{\epsilon}g_0^{(n)})$, respectively, and conclude that (\ref{eqn_gg0}) holds throughout $\overline{B}$, hence in particular in $B(0,r)$, as required.

Since $(g_i)$ is a solution to a specific elliptic PDE, the $C^0$-convergence that we just established can be improved to $C^2$-convergence by the following standard argument. Write $g_i=e^{u_i}|\dz|^2$ and $g_0^{(n)}=e^{u_0}|\dz|^2$, so that Wang's equation can be written as the following equation satisfied by $u_i\in C^\infty(B(0,r_i))$ and $u_0\in C^\infty(\mathbb{C})$:
$$
\Delta u=2(e^{u}-2|z|^{2n}e^{-2u}).
$$
Fix $p>2$, $\alpha\in(0,1)$. The $C^0$-convergence implies that the $L^p$-norm of $u_i$ on $B(0,r)$ is uniformly bounded, so the $L^p$ estimate for second order linear elliptic PDEs (\cite[Theorem 9.11]{gilbarg-trudinger}) implies that the $W^{2,p}$-norm (on any smaller disk) is uniformly bounded, hence so is the $C^1$-norm by Sobolev embedding. The Schauder estimate (\cite[Theorem 6.2]{gilbarg-trudinger}) then implies that the $C^{2,\alpha}$-norm of $u_i$ is also uniformly bounded. This allows us to apply Arzel\`a-Ascoli and conclude that any subsequence of $(u_i)$ has a further subsequence $C^2$-converging to $u$, which means the sequence $(u_i)$ itself $C^2$-converges to $u_0$.
\end{proof}

\subsection{Continuous dependence of ODE system on parameters}\label{subsec_ode}
The ODE result that we need in order to deduce Proposition \ref{prop_developingconvergence} from Proposition \ref{prop_metricconvergence} is:
\begin{lemma}\label{lemma_ode}
Let $T>0$, $U$ be a set, $\big(A_i(t,\mu)\big)_{i=1,2,\cdots}$ be a sequence of bounded $\mathfrak{gl}_n\mathbb{R}$-valued functions on $[0,T]\times U$ uniformly converging to a function $A_\infty(t,\mu)$, and $Y_j$ ($j\in\mathbb{N}\cup\{\infty\}$) be a $\GL(n,\mathbb{R})$-valued function on $[0,T]\times U$ satisfying
$$
\frac{\pa}{\pa t}Y_j(t,\mu)+A_j(t,\mu)Y_j(t,\mu)=0,\quad Y_j(0,\mu)=I.
$$
Then $Y_i$ converges to $Y_\infty$ uniformly on $[0,T]\times U$ as $i\to\infty$.
\end{lemma}
This should be viewed as a continuous parameter-dependence property for solutions to linear ODE systems. The proof is standard and is given below for the sake of completeness. The main ingredient is the following classical lemma:
\begin{lemma}[{\textbf{Gronwall Inequality}. See \cite[p.24]{hartman}}]\label{lemma_gronwall}
	Let $x(t)$ be a continuous nonnegative function on $[0,T]$ satisfying
	$x(t)\leq C_1+C_2\int_0^tx(s)\dif s$ for all $t\in[0,T]$, where  $C_1,C_2>0$ are constants. Then $x(t)\leq C_1e^{C_2t}$ for all $t$.
\end{lemma}

\begin{proof}[Proof of Lemma \ref{lemma_ode}]
Given a norm $\|\cdot\|$ on $\mathfrak{gl}_n\mathbb{R}$ with $\|AB\|\leq \|A\|\cdot\|B\|$, we have 
%The initial value condition $Y_i(0,\mu)=Y_\infty(0,\mu)=I$ implies
\begin{align*}
&\|Y_i(t,\mu)-Y_\infty(t,\mu)\|=\left\|\int_0^t\left[A_\infty(s,\mu)Y_\infty(s,\mu)-A_i(s,\mu)Y_i(s,\mu)\right]\dif s\right\|\\
%\leq& \int_0^t\left\|A_0(s,\mu)X_0(s,\mu)-A_i(s,\mu)X_i(s,\mu)\right\|\dif s \\
%\leq& \int_0^t\left\|A_0(s,\mu)X_0(s,\mu)-A_i(s,\mu)X_0(s,\mu)\right\|\dif s\\
%&+\int_0^t\left\|A_i(s,\mu)X_0(s,\mu)-A_i(s,\mu)X_i(s,\mu)\right\|\dif s\\
%\leq& \int_0^t\|A_0(s,\mu)-A_i(s,\mu)\|\cdot\|X_0(s,\mu)\|\dif s\\
%&+\int_0^t\|A_i(s,\mu)\|\cdot\|X_i(s,\mu)-X_0(s,\mu)\|\dif s
&\leq \int_0^t\left\|A_\infty Y_\infty-A_iY_i\right\|\dif s \leq \int_0^t\left\|A_\infty Y_\infty-A_iY_\infty\right\|\dif s+\int_0^t\left\|A_iY_\infty-A_iY_i\right\|\dif s\\
&\leq \int_0^t\|A_\infty-A_i\|\cdot\|Y_\infty\|\dif s+\int_0^t\|A_i\|\cdot\|Y_i-Y_\infty\|\dif s
\end{align*}
(where every function in the second and third lines is evaluated at $(s,\mu)$). 
If $\|Y_\infty\|$ is bounded on $[0,T]\times U$, we can apply Lemma \ref{lemma_gronwall} to get
$$
\|Y_i(t,\mu)-Y_\infty(t,\mu)\|\leq C\sup_{[0,T]\times U}\|A_i-A_\infty\|,
$$
%where the constant $C$ only depends on $T$ and the upper bound of $\|Y_\infty\|$ and the $\|A_i\|$'s. 
which implies the required uniform convergence. But the boundedness of $\|Y_\infty\|$ itself also follows from Lemma \ref{lemma_gronwall}  because
\begin{align*}
	\|Y_\infty(t,\mu)\|&\leq \|I\|+\|Y_\infty(t,\mu)-I\|= 1+\left\|\int_0^tA_\infty(s,\mu)Y_\infty(s,\mu)\dif s\right\|\\
	&\leq 1+\int_0^t\|A_\infty\|\cdot \|Y_\infty\|\dif s\leq 1+\sup_{[0,T]\times U}\|A_\infty\|\cdot \int_0^t\|Y_\infty\|\dif s.
\end{align*}
\end{proof}
\subsection{Proof of Proposition \ref{prop_developingconvergence}}\label{subsec_poof}
\begin{proof}
Denote $\Omega_i:=\P(f(X))$, $\Omega:=\P(f(\mathbb{C}))$ and
consider their dual domains $\Omega_i^*,\Omega^*\subset\mathbb{RP}^{2*}$ (see \S \ref{subsec_duality}).
The required convergence $\Omega_i\to \Omega$ is equivalent to the following condition:
\begin{equation}\label{eqn_convergencecondition}
\text{$\forall$ compact $C\subset\Omega$, $C'\subset\Omega^*$, $\exists i_0\in\mathbb{N}$ such that $i\geq i_0$ $\Rightarrow$ $C\subset \Omega_i$, $C'\subset\Omega_i^*$.} 
\end{equation}
%that for any compact sets $E\subset \Omega$ and $F\subset\Omega^*$, we have inclusions $E\subset\Omega_i$ and $F\subset\Omega_i^*$ when $i$ is sufficiently large.

Let $f^*$ be the affine spherical immersion dual to $f$ (see Proposition \ref{prop_wangdeveloping}), and let $\P f:\mathbb{C}\to\RP$ and $\P f^*:\mathbb{C}\to\mathbb{RP}^{2*}$ be the projectivizations of $f$ and $f^*$, respectively. To prove \eqref{eqn_convergencecondition}, we fix compact subsets $C\subset\Omega=\P f(\mathbb{C})$, $C'\subset\Omega^*=\P f^*(\mathbb{C})$ and take a big enough disk $B(0,r)\subset\mathbb{C}$ such that its images by $\P f$ and $\P f^*$ contains $C$ and $C'$, respectively. Viewing the restrictions $f_i|_{U_i}$ and $f_i^*|_{U_i}$ as maps from $B(0,r_i)$ to $\R^3$ and $\R^{3*}$, respectively, via the identification (\ref{eqn_diskneighborhood}) of $U_i$ with $B(0,r_i)$, we claim that they converge uniformly on the closed disk $\overline{B}(0,r)$ to $f$ and $f^*$, respectively, as $i\to\infty$.

The claim follows from Proposition \ref{prop_metricconvergence} and Lemma \ref{lemma_ode}. Here we give a detailed argument only for $f_i|_{U_i}$, as the case of $f^*_i|_{U_i}$ is the similar. Under the coordinate $z=x+\ima y$ of $B(0,r)$, we write $g_i=e^{u_i}|\dz|^2$, $g_0^{(n)}=e^u|\dz|^2$ and view the affine sphere connections $\nabla_i:=\nabla_{\ve{\psi}_i}$ and $\nabla:=\nabla_{z^n\dz^3}$ (see \S \ref{subsec_wang}) locally as flat connections on the vector bundle $\T B(0,r)\oplus\underline{\R}$. They are expressed under the frame $(\pa_x,\pa_y,\underline{1})$ as
$\nabla_i=d+A_i$ and $\nabla=d+A$, where $A_i$ (resp. $A$) is a matrix of $1$-forms whose coefficients are linear combinations of $e^{\pm u_i}$ (resp $e^u$) and the first order derivatives of $u_i$ (resp. $u$). Therefore, the $C^1$-convergence $g_i\to g_0^{(n)}$ from Proposition \ref{prop_metricconvergence} implies the uniform convergence
\begin{equation}\label{eqn_aconvergence}
A_i\to A \text{ on $B(0,r)$ as }i\to\infty.
\end{equation}

Given $z\in\overline{B}(0,r)$, consider the parallel transport  $\para_i(z)$ of $\nabla_i$ from $z$ to $0$, viewed as a matrix in $\SL(3,\R)$ under the frame $(\pa_x,\pa_y,\underline{1})$. Also let $\para(z)$ be the same parallel transport of $\nabla$. Then $f_i$ and $f$ have coordinate expressions (see Definition \ref{def_wangdeveloping})
$$
f_i(z)=\para_i(z)
\begin{pmatrix}
0\\[-3pt]0\\[-3pt]1
\end{pmatrix},\quad
f(z)=\para(z)
\begin{pmatrix}
0\\[-3pt]0\\[-3pt]1
\end{pmatrix}\quad\text{for all $z\in\overline{B}(0,r)$.}
$$
To prove the claim, we view $\para_i(z)$ and $\para(z)$ as solutions of parametrized linear ODE systems. Namely, by definition of parallel transport, for all $t\in[0,r]$ and $\theta\in[0,2\pi]$ we have
$$
\frac{\dif}{\dif t}\para_i(te^{\theta\ima})=\para_i(te^{\theta\ima})A_i(\pa_{t}(te^{\theta\ima})),\quad \frac{\dif}{\dif t}\para(te^{\theta\ima})=\para(te^{\theta\ima})A(\pa_{t}(te^{\theta\ima})).
$$ 
Here, the matrix valued $1$-forms $A_i$ and $A$ are evaluated at each tangent vector $\pa_{t}(te^{\theta\ima})=e^{\theta\ima}\in\T_{te^{\theta\ima}}B(0,r)$, yielding matrix-valued functions $(t,\theta)\mapsto A_i(\pa_{t}(te^{\theta\ima}))$ and $(t,\theta)\mapsto A(\pa_{t}(te^{\theta\ima}))$  on $[0,r]\times[0,2\pi]$. By (\ref{eqn_aconvergence}), the former converges uniformly to the latter, so we can apply Lemma \ref{lemma_ode} and conclude that $\para_i(z)$ converges to $\para(z)$ uniformly for $z\in\overline{B}(0,r)$, whence the claim follows.

To deduce the required property \eqref{eqn_convergencecondition} from the claim, we need  two basic facts:
\begin{itemize}
	\item Wang's developing maps $f_i$ and $f$ take values in a closed affine half-space of $\R^3$ not containing $0$. This is explained at the end of \S \ref{subsec_wang}.
	\item If we fix a Euclidean metric on $\R^3$ and endow $\RP$ with the metric as the antipodal quotient of the unit sphere $\mathbb{S}^2\subset\mathbb{R}^3$, then the projection $\P:\mathbb{R}^3\setminus\{0\}\to\RP$ is Lipschitz when restricted to the complement of a neighborhood of $0$. In particular, $\P$ is Lipschitz on the aforementioned half-space of $\R^3$ where $f_i$ and $f$ take values.
\end{itemize}
With these facts, we deduce from the claim that $\P f_i$ converges to $\P f$ uniformly on $\overline{B}(0,r)$ with respect to the metric on $\RP$. Therefore, $\P f_i(\overline{B}(0,r))$ contains the compact set $C\subset \P f(B(0,r))$ when $i$ is big enough, and the same holds for $\P f_i^*(\overline{B}(0,r))$ and $C'$. This implies the required property \eqref{eqn_convergencecondition} and completes the proof.
\end{proof}

\section{Limit domains are regular polygons}\label{sec_limitconvex}
We prove in this section the ``only if'' part of Theorem \ref{thm_intro2}.
\subsection{Pointed Gromov-Hausdorff convergence of Blaschke metric}\label{subsec_pointedgh}
The Blaschke metric has the following fundamental continuity property:
\begin{lemma}[Benoist-Hulin]\label{lemma_benoisthulin}
	Let $(\Omega_i)_{i=1,2,\cdots}$ be a sequence in $\hC$ converging to some $\Omega\in\hC$ and $K\subset\RP$ be a compact set contained in $\Omega$ and every $\Omega_i$. Then the Blaschke metric $g_{\Omega_i}$ and Pick differential $\ve{\psi}_{\Omega_i}$ converge uniformly on $K$ to $g_\Omega$ and $\ve{\psi}_\Omega$, respectively.
\end{lemma}
The original result claims $C^k$-convergence in $K$, and is deduced from the fact that the affine sphere $\Sigma_{\Omega_i}$ (see \S \ref{subsec_affinesphere}) compactly $C^k$-converges to $\Sigma_\Omega$. See \cite[Corollary 3.3]{benoist-hulin} and \cite[Theorem 4.4]{dumas-wolf} for details. In this section we only make use of the statement for $g_{\Omega_i}$, while the one for $\ve{\psi}_{\Omega_i}$ is used in \S \ref{subsec_at} below.

A \emph{pointed metric space} $(M,d,p)$ consists of a metric space $(M,d)$ and a choice of a base point $p\in M$, while the \emph{Gromov-Hausdorff convergence} $(M_i,d_i,p_i)\to (M,d,p)$ of pointed metric spaces rough means that any metric ball in $M$ centered at $p$ is approximated by metric balls in $M_i$ centered at $p_i$. More precisely:

\begin{definition}[{\cite[Definition 8.1.1]{bbi}}]\label{def_gh}
%A map between metric spaces $f:(M_1,d_1)\to(M_2,d_2)$ is called an \emph{almost isometry} with \emph{error} $\epsilon$ if 
%$|d_1(x,y)-d_2(f(x),f(y))|\leq\epsilon$ for all $x,y\in M_1$,
%and every point in $M_2$ has distance at most $\epsilon$ from the image $f(M_1)$.  
A sequence $(M_i,d_i,p_i)_{i=1,2,\cdots}$ of pointed metric spaces is said to \emph{converge} to a pointed metric space $(M,d,p)$ if given any $R>0$ and $\epsilon>0$, there is $i_0\in\mathbb{N}$ such that for every $i\geq i_0$, there exists a (not necessarily continuous) map 
$f_i$ from the ball $B_{d_i}(p_i,R)\subset M_i$ to $M$, satisfying:
\begin{itemize}
	\item $f_i(p_i)=p$;
	\item $\sup_{x,y\in B_{d_i}(p_i,R)}|d_i(x,y)-d(f_i(x),f_i(y))|<\epsilon$;
	\item the $\epsilon$-neighborhood of the image $f_i(B_{d_i}(p_i,R))$ contains $B_d(p,R-\epsilon)$.
\end{itemize}
\end{definition}
\begin{remark}\label{remark_gh}
We mainly use the following sufficient condition for the convergence $(M_i,d_i,p_i)\to (M,d,p)$, when $M_i$, $M$ are manifolds and $d_i$, $d$ are given by Riemannian metrics $g_i$, $g$ such that every pair of points are joint by a geodesic and $g$ is complete: For each $i$, there is a neighborhoods $U_i$ of $p_i$ in $M_i$, a neighborhood $V_i$ of $p$ in $M$ and a diffeomorphism $f_i:U_i\overset\sim\to V_i$, such that 
\begin{itemize}
	\item $V_i$ contains a metric ball centered at $p$ with radius tending to $\infty$ as $i\to\infty$;
	\item $d(p, f_i(p_i))\to 0$ as $i\to\infty$;
	\item the push-forward metric $f_i^*g_i|_{U_i}$ on $V_i$ converges to $g$ uniformly on any compact subset of $M$ (which is contained in $V_i$ when $i$ is big enough).
\end{itemize}
We leave to the reader the proof of the fact that this condition implies the convergence, only mentioning here that a key point in the proof is to show $U_i$ also contains a metric ball centered at $p_i$ with radius tending to $\infty$, which relies on the completeness of $(M,d)$.
\end{remark}
Given a sequence $(\Omega_i)$ in $\C$ converging to $\Omega$, taking an exhausting sequence of open sets $V_1\subset V_2\subset\cdots \subset\Omega$ such that the closure of $V_i$ is contained in both $\Omega_i$ and $\Omega$, we deduce From Lemma \ref{lemma_benoisthulin} and the above remark:
\begin{corollary}\label{coro_gh}
	Let $(\Omega_i)_{i=1,2,\cdots}$ be a sequence in $\C$ converging to some $\Omega\in\C$, and $p\in\RP$ be a point contained in $\Omega$ and every $\Omega_i$. Then $(\Omega_i,g_{\Omega_i}, p)$ converges to $(\Omega,g_\Omega,p)$.
\end{corollary}
This gives a necessarily condition for the convergence $\Omega_i\to\Omega$ in the space $\C$, which we use to prove the ``only if'' part of Theorem \ref{thm_intro2} in the next subsection. 
\begin{remark}
The completeness of $g_\Omega$ plays an important role in Corollary \ref{coro_gh}. Namely, the statement would not hold if we replace $g_{\Omega_i}$ and $g_\Omega$ by the singular flat metrics $|\ve{\psi}_{\Omega_i}|^\frac{2}{3}$ and $|\ve{\psi}_{\Omega}|^\frac{2}{3}$, which are not necessarily complete.
For example, let $\Omega$ be such that the metric $|\ve{\psi}_\Omega|^\frac{2}{3}$ has bounded diameter and $(\Omega_i)$ be a sequence of polygons approximating $\Omega$, then $(\Omega_i,\ve{\psi}_{\Omega_i})\cong (\mathbb{C},\psi_i(z)\dz^3)$ for some polynomial $\psi_i(z)$ by \cite{dumas-wolf} and $(\Omega_i,|\ve{\psi}_{\Omega_i}|^\frac{2}{3}, p)$ cannot converge to $(\Omega,|\ve{\psi}_\Omega|^\frac{2}{3},p)$.
\end{remark}

\subsection{Proof of the ``only if'' part of Theorem \ref{thm_intro2}}
A fundamental property of the above notion of convergence for pointed metric spaces is that the limit is unique in the category of pointed complete length spaces (see \cite[Theorem 8.1.7]{bbi}). As a consequence, if there are two subsequences converging to  pointed complete length spaces not isometric to each other, then the original sequence does not converge.  
Using this fact and Proposition \ref{prop_metricconvergence}, we determine all possible limits of $(\Omega_i,g_{\Omega_i},p_i)$ for a sequence $(\Omega_i)$ in $\C$ satisfying the assumptions of Theorem \ref{thm_intro2}:
\begin{proposition}\label{prop_gh}
	For $i=1,2,\cdots$, let $X_i$ be a contractible Riemann surface endowed with a holomorphic cubic differential $\ve{\psi}_i$. Denote the flat metric underlying  $\ve{\psi}_i$ by $h_i:=|\ve{\psi}_i|^\frac{2}{3}$ and assume that the conditions in Theorem \ref{thm_intro2} are fulfilled, namely, $h_i$	is complete, the minimal distance (with respect to $h_i$) between the zeros of $\ve{\psi}_i$ tends to $+\infty$ as $i\to\infty$, and the multiplicities of zeros are uniformly bounded. Let $g_i$ denote the complete solution to Wang's equation on $(X_i,\ve{\psi}_i)$ (see \S \ref{subsec_blaschke}) and pick a point $p_i\in X_i$ for each $i$. 
	\begin{enumerate}
		\item\label{item_gh1} If $d_{h_i}(p_i,Z(\ve{\psi}_i))\to+\infty$ as $i\to\infty$ (where $d_{h_i}$ is the distance induced by $h_i$ and $Z(\ve{\psi}_i)$ is the set of zeros of $\ve{\psi}_i$), then $(X_i,g_i,p_i)$ converges to $(\mathbb{C},|dz|^2,0)$.
		\item\label{item_gh2} If there is a positive integer $n$ such that for every $i$ apart from finitely many exceptions, there is a zero $z_i\in Z(\ve{\psi}_i)$ of multiplicity $n$ satisfying
		$$
		d_{h_i}(p_i,z_i)\to R_0\in[0,+\infty)\ \text{ as }i\to\infty,
		$$
		then $(X_i,g_i,p_i)$ converges to $(\mathbb{C},g_0^{(n)},x_0)$, where $x_0$ is the point in $[0,+\infty)\subset\mathbb{C}$ with distance $R_0$ from the origin $0\in\mathbb{C}$, with respect to the flat metric $|z|^\frac{2n}{3}|\dz|^2$ underlying the cubic differential $z^n\dz^3$.
		\item\label{item_gh3} Otherwise, $(X_i,g_i,p_i)$ does not converge.
	\end{enumerate}
\end{proposition}
\begin{proof}
%	(\ref{item_gh1}) Denote $R_i:=d_{h_i}(p_i,Z(\ve{\psi}_i))$. Then there is a $(B(0,R_i),p)$   (\ref{item_gh2})  (\ref{item_gh3})
%Parts (\ref{item_gh1}) and (\ref{item_gh2}) follow from Proposition \ref{prop_metricconvergence}: 

Under the assumption of Part (\ref{item_gh1}), there exist $r_i>0$ and a neighborhood $U_i$ of $p_i$ for each $i$ such that $r_i\to+\infty$  and  
$$
(U_i,\ve{\psi}_i,p_i)\cong(B(0,r_i),\,\dz^3,0).
$$ 
Similarly, under the assumption of (\ref{item_gh2}) there are $r_i\to+\infty$, $U_i$ and $x_i\in[0,r_i)\subset B(0,r_i)$ converging to $x_0$, such  that
$$
(U_i,\ve{\psi}_i,p_i)\cong(B(0,r_i),\,z^n\dz^3,x_i).
$$
Therefore, Proposition \ref{prop_metricconvergence} and the sufficient condition for pointed Gromov-Hausdorff convergence in Remark \ref{remark_gh} imply the required statements (\ref{item_gh1}) and (\ref{item_gh2}).

We proceed to prove Part (\ref{item_gh3}). Since the multiplicities of zeros are uniformly bounded, 
	 if the distance $d_{h_i}(p_i,Z(\ve{\psi}_i))$ is unbounded but does not tends to $+\infty$, then there are two subsequence $(X_i, \ve{\psi}_i, p_i)_{i\in I}$ and $(X_i, \ve{\psi}_i, p_i)_{i\in J}$ satisfying the hypotheses of (\ref{item_gh1}) and (\ref{item_gh2}), respectively. The aforementioned uniqueness property of limit then implies that the whole sequence $(X_i,g_i,p_i)_{i\in\mathbb{N}}$ does not converge, as its two subsequences have different limits by Statements (\ref{item_gh1}) and  (\ref{item_gh2}).

Similarly, if $d_{h_i}(p_i,Z(\ve{\psi}_i))$ is bounded but the hypothesis of (\ref{item_gh2}) is not satisfied, then there are two subsequences, both satisfying the hypothesis but either with different $n$ or $R_0$. But two pointed metric spaces of the form $(\mathbb{C},g_0^{(n)},x_0)$ with different $n$ or $x_0$ are not isometric to each other (this follows \eg from the uniqueness of solutions to Wang's equation (see \S \ref{subsec_blaschke}), which implies that an isometry $\mathrm{Isom}^+(\mathbb{C},g_0^{(n)})$ ($n\geq1$) must preserve the cubic differential $z^n\dz^3$ up to multiplication by some $\lambda\in\mathbb{C}$, $|\lambda|=1$, hence must be a rotation).
Therefore, $(X_i,g_i,p_i)$ does not converge, and the proof is completed.
\end{proof}
We can now finish the proof of Theorem \ref{thm_intro2}:
\begin{proof}[Proof of the ``only if'' part of Theorem \ref{thm_intro2}]
	Let $(\Omega_i)$ be a sequence in $\C$ satisfying the assumptions of the theorem and converging to some $\Omega\in\C$, and pick $p\in\Omega$. Restricting to a subsequence $(\Omega_i)_{i\geq i_0}$ if necessary, we may assume every $\Omega_i$ contains $p$, so that the pointed metric space $(\Omega_i,g_{\Omega_i},p)$ converges to $(\Omega,g_\Omega,p)$ by Corollary \ref{coro_gh}. By assumptions on $\ve{\psi}_{\Omega_i}$, Proposition \ref{prop_gh} implies that $(\Omega,g_\Omega)$ is isometric to $(\mathbb{C},|\dz|^2)$ or some $(\mathbb{C},g_0^{(n)})$ ($n\geq1$), and in the latter case every $\ve{\psi}_{\Omega_i}$, apart from at most finitely many exceptions, has a zero $z_i\in\Omega_i$ of order $n$. The uniqueness of complete solutions to Wang's equation (see \S \ref{subsec_blaschke}) then implies that $(\Omega,\ve{\psi}_{\Omega})$ is isomorphic to $(\mathbb{C},\dz^3)$ or $(\mathbb{C},z^n\dz^3)$. By \cite{dumas-wolf}, this means $\Omega$ is a triangle or a regular $(n+3)$-gon, completing the proof of the theorem.
\end{proof}

\section{Pick differentials of domains generated by reflections}\label{sec_tits}
In this section, we first review the construction of properly convex domains generated by projective deformations of hyperbolic triangular reflection groups, then we describe the Pick differentials of these domains and prove Theorem \ref{thm_intro3}.
\subsection{Projective reflection and Pick differential}
A projective transformation $\sigma\in\SL(3,\R)$ is called a \emph{reflection} if it is conjugate to
$$
\begin{pmatrix}
1&&\\[-3pt]
&-1&\\[-3pt]
&&-1
\end{pmatrix},
$$
or equivalently, $\sigma$ has order $2$ and pointwise fixes a projective line $L\subset\RP$ and a point $p\in\RP\setminus L$. 
We call $\sigma$ a reflection \emph{across} the line $L$.
 
A properly convex domain $\Omega\subset\RP$ preserved by $\sigma$ has two possible relative positions with respect to the $L$ and $p$: either $\Omega$ is a bounded convex domain in the affine chart $\RP\setminus L\cong \R^2$ and contains $p$, or $L$ passes through $\Omega$ and $p$ is not in $\Omega$. The action of $\sigma$ on $\Omega$ is orientation preserving and reversing in the former and latter case, respectively. We are mainly interest in the latter case, where $\sigma$ brings the Pick differential $\ve{\psi}_\Omega$ to its conjugate, as the following proposition shows. Here, $\ve{\psi}_\Omega$ is viewed as a tri-linear map $\T_p\Omega\times\T_p\Omega\times\T_p\Omega\to\mathbb{C}$ for each tangent space $\T_p\Omega$ and put $\ve{\psi}_\Omega(v):=\ve{\psi}_\Omega(v,v,v)$ for simplicity.
\begin{proposition}\label{prop_conjugate}
Let $\Omega\subset\RP$ be a properly convex domain preserved by a reflection $\sigma$ across a line $L$ passing through $\Omega$. Endow $\Omega$ with an orientation and let $\ve{\psi}_\Omega$ be its Pick differential. Then for any point $p\in\Omega$ and tangent vector $v\in\T_p\Omega$, we have 
$$
\ve{\psi}_\Omega(v)=\overline{\ve{\psi}_\Omega(\dif\sigma_p(v))}.
$$
In particular, $L$ is a \emph{real trajectory} of $\ve{\psi}_\Omega$ in the sense that  $\ve{\psi}_\Omega(v)\in\R$ for every tangent vector $v$ of $L$.
\end{proposition}
\begin{proof}
Let $-\Omega$ denote $\Omega$ with the reversed orientation (\cf \S \ref{subsec_affinesphere}). Then $\sigma:\Omega\to -\Omega$ is an orientation preserving isomorphism. By the properties of Pick differentials discussed in \S \ref{subsec_affinesphere}, we have
$$
\ve{\psi}_\Omega=\sigma^*\ve{\psi}_{-\Omega}=\sigma^*\,\overline{\ve{\psi}_{\Omega}}~,
$$
which is just another formulation of the required equality.
\end{proof}

\subsection{$\frac{1}{3}$-translation surface with real boundary}\label{subsec_real}
Instead of viewing a cubic differential $\ve{\psi}$ as a holomorphic object on a Riemann surface $S$, we shall henceforth consider the pair $\ve{S}=(S,\ve{\psi})$, when $\ve{\psi}$ is nontrivial, as a \emph{$\frac{1}{3}$-translation surface} (with singularities), in the same way as how surfaces endowed with abelian or quadratic differentials are viewed as \emph{translation} or \emph{half-translation} surfaces. More precisely, around every ordinary points, $\ve{S}$ is locally modeled on $(\mathbb{C},\dz^3)$, hence in particular carries a flat metric, whereas each zero of multiplicity $n$ is a conical singularity for the flat metric, with cone angle 
$$
\theta=2\pi+\tfrac{2\pi}{3}n.
$$ 
It is important to note that every singularity has cone angle bigger than $2\pi$.

%can be constructed by gluing together one or several polygonal regions in $\mathbb{C}$ along their sides through maps of the form $z\mapsto e^{2k\pi\ima/3}z+z_0$ ($k=0,1,2$, $z_0\in\mathbb{C}$). 

An advantage of this point of view is that the notion naturally extends to surfaces with boundary. 
More specifically, Proposition \ref{prop_conjugate} motivates us to consider $\frac{1}{3}$-translation surfaces with boundary obtained by cutting one without boundary along real trajectories. We define such surfaces precisely as follows:
 
\begin{definition}[\textbf{$\frac{1}{3}$-translation surface with real boundary}]\label{def_realboundary}
A $\frac{1}{3}$-translation surface with \emph{real boundary} is a $\frac{1}{3}$-translation surface $\ve{S}$ with boundary $\pa\ve{S}$, such that any local chart of $\ve{S}$ around a boundary point identifies $\pa\ve{S}$ locally with a piecewise line segment in $(\mathbb{C},\dz^3)$ which has locally finitely many pieces and each piece is a real trajectory (\ie parallel to one of the three lines $\R$, $e^{\frac{2\pi\ima}{3}}\R$ and $e^\frac{4\pi\ima}{3}\R$ in $\mathbb{C}$). A junction point $q\in\pa\ve{S}$ of two pieces is referred to as a \emph{corner} of $\ve{S}$.
\end{definition}
We will need a Gauss-Bonnet formula for such surfaces. Recall that the formula for a compact surface $S$ (with boundary) with a smooth Riemannian metric is
$$
\int_S\kappa\,\dif A+\int_{\pa S}\kappa_g\dif s=2\pi\chi(S),
$$
where $\chi(S)$ is the Euler characteristic, $\kappa:S\setminus\pa S\to\R$ and $\kappa_g:\pa S\to \R$ are the curvature of the interior and the geodesic curvature of the boundary, respectively, whereas $\dif A$ and $\dif s$ are the area and arc-length measures. If $S$ is a flat surface with conical singularities, although $\kappa$ is not well-defined as a function,
the curvature measure $\kappa\,\dif A$ is still naturally defined as the atomic measure supported at the singularities such that a singularity with cone angle $\theta$ contributes a mass of $2\pi-\theta$. In particular, for a $\frac{1}{3}$-translation surface $\ve{S}$, by the aforementioned relation between $\theta$ and multiplicity of zero, we can write
$$
\kappa\,\dif A=-\tfrac{2\pi}{3}\sum_{p\in Z(\ve{S})}n(\ve{S},p)\delta_p,
$$
where $Z(\ve{S})$ is the set of zeros (in the interior of $\ve{S}$), $\delta_p$ is the Dirac measure at $p$, and $n(\ve{S},p)$ is the multiplicity of the zero $p$. Analogously, when $\ve{S}$ has real boundary, we can make sense of the boundary geodesic curvature measure as
$$
\kappa_g\dif s=-\tfrac{\pi}{3}\sum_{q\in C(\ve{S})}l(\ve{S},q)\delta_q,
$$
where $C(\ve{S})\subset\pa\ve{S}$ is the set of corners, and the integer $l(\ve{S},q)$ is defined as follows:
\begin{definition}[\textbf{Index of a corner}]
Given a $\frac{1}{3}$-translation surface $\ve{S}$ with real boundary, for each corner $q\in\pa\ve{S}$, we let $\angle(\ve{S},q)$ denote the interior angle of $\ve{S}$ at $q$ and define the \emph{index} $l(\ve{S},q)\in\mathbb{Z}\cap[-2,+\infty)$ by the equality
$$
\angle(\ve{S},q)=\pi+\tfrac{\pi}{3}l(\ve{S},q).
$$
\end{definition}
Note that $l(\ve{S},q)$ is an integer and its smallest possible value is $-2$ because $\angle(\ve{S},q)$ is a positive integer multiple of $\frac{\pi}{3}$. We also have $l(\ve{S},q)\neq0$, because otherwise $q$ would not be a corner. Some examples of $\frac{1}{3}$-translation disks with real boundary are given in Figure \ref{figure_real}, where the indices of corners are marked.
\begin{figure}[h]
	\labellist
	\pinlabel {$-2$} at -15 65
	\pinlabel {$-2$} at 50 15
	\pinlabel {$-2$} at 110 280
	\pinlabel {$-2$} at 250 15
	\pinlabel {$-2$} at 365 60
	\pinlabel {$-2$} at 434 22
    \pinlabel {$-2$} at 645 22
	\pinlabel {$-2$} at 503 293
	\pinlabel {$-2$} at 500 159
	\pinlabel {$-2$} at 768 105
	\pinlabel {$-2$} at 910 368
	\pinlabel {$-2$} at 800 -13
	\pinlabel {$-2$} at 1033 -13
	\pinlabel {$2$} at 110 110
	\pinlabel {$4$} at 540 76
	\endlabellist
	\centering
	\includegraphics[width=10cm]{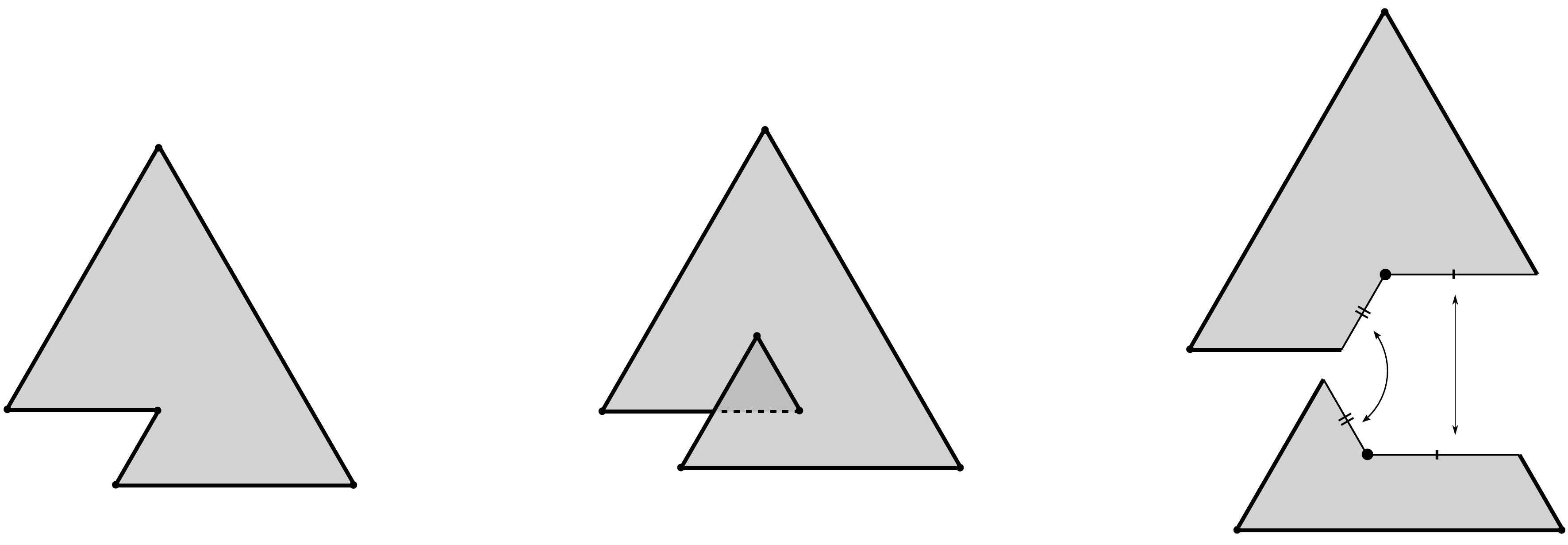}
	\caption{Three $\frac{1}{3}$-translation disks with real boundary, with the indices of the corners marked. The first two examples do not have zeros in the interior. The last one is constructed from gluing, and the two black dots give a simple zero in the interior.}
	\label{figure_real}
\end{figure}
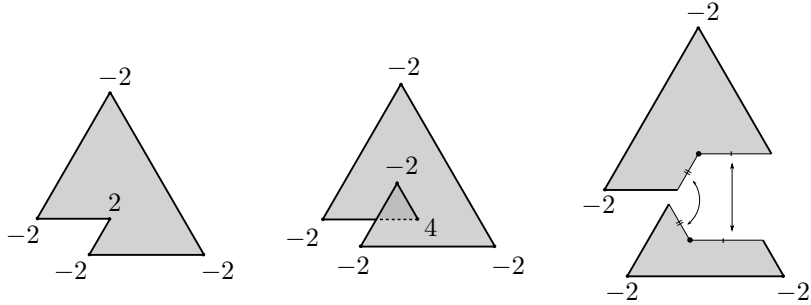
%For a \emph{compact} $\frac{1}{3}$-translation surfaces $\ve{S}$ with boundary, the boundary $\pa\ve{S}$ is topologically a disjoint union of circles and there are only finitely many $q\in\pa\ve{S}$ with exterior angle $\theta(\ve{S},q)\neq0$. In this case, we have the following fundamental relations between the exterior angles, the zeros and the topology of $\ve{S}$:

As mentioned, $l(\ve{S},q)$ is an analogue of the multiplicity $n(\ve{S},p)$ and appears in the boundary term of the Gauss-Bonnet formula. The formula in this case is:
\begin{theorem}[\textbf{Gauss-Bonnet}]\label{thm_gaussbonnet}
For any compact $\frac{1}{3}$-translation surface $\ve{S}$ with real boundary, we have (see the preceding paragraphs for the notations)
\begin{equation}\label{eqn_gb}
-\frac{2\pi}{3}\sum_{p\in Z(\ve{S})}n(\ve{S},p)-\frac{\pi}{3}\sum_{q\in C(\ve{S})}l(\ve{S},q)=2\pi\chi(\ve{S}).
\end{equation}
\end{theorem}
One readily verifies the theorem for the examples in Figure \ref{figure_real}. The theorem easily generalized to flat surfaces with conical singularities and piecewise geodesic boundary, and the generalized version can be proved by decomposing the surface into Euclidean triangles, similarly as in the proof of classical Gauss-Bonnet formula through triangulation (see \eg \cite{docarmo}). We omit the details.

Finally, we define the \emph{conjugate} $\overline{\ve{S}}$ of a $\frac{1}{3}$-translation surface $\ve{S}$ as the $\frac{1}{3}$-translation surface obtained by composing every $(\mathbb{C},\dz^3)$-valued local chart of $\ve{S}$ with the conjugation $\mathbb{C}\to\mathbb{C}$, $z\mapsto\bar{z}$. 
%Analytically, if $\ve{S}=(S,\ve{\psi})$ then $\overline{\ve{S}}=(\overline{S},\overline{\ve{\psi}})$, where $\overline{S}$ is the Riemann surface $S$ endowed with the opposite complex structure, under which the conjugate cubic differential $\overline{\ve{\psi}}$ is holomorphic. 
In other words, if $\ve{S}$ is as in Figure \ref{figure_real}, then $\overline{\ve{S}}$ is the mirror image obtained by reflecting $\ve{S}$ across a real trajectory in $(\mathbb{C},\dz^3)$. With the terminologies introduced in this subsection, we deduce from Proposition \ref{prop_conjugate}: 
\begin{corollary}\label{coro_flip}
Let $\Omega$, $L$ and $\sigma$ be as in Proposition \ref{prop_conjugate} and $U\subset\Omega$ be an open set as in Figure \ref{figure_reflection}, whose boundary meets $L$ along a segment $I$. Then $U\cup I$ (as a subsurface of $(\Omega,\ve{\psi}_\Omega)$) is a $\frac{1}{3}$-translation surface with real boundary, and $\sigma(U)\cup I$ is its conjugate. For each corner $q$ of $U\cup I$ in the interior of $I$, the angle $\angle(U\cup I, q)$ is bigger than $\pi$.
\end{corollary}
	\begin{figure}[h]
	\labellist
	\pinlabel {$\Omega$} at 40 50
	\pinlabel {$I$} at 235 125
	\pinlabel {$U$} at 160 105
	\pinlabel {$\sigma(U)$} at 310 105
	\pinlabel {$L$} at 240 265
	\endlabellist
	\centering
	\includegraphics[width=3.7cm]{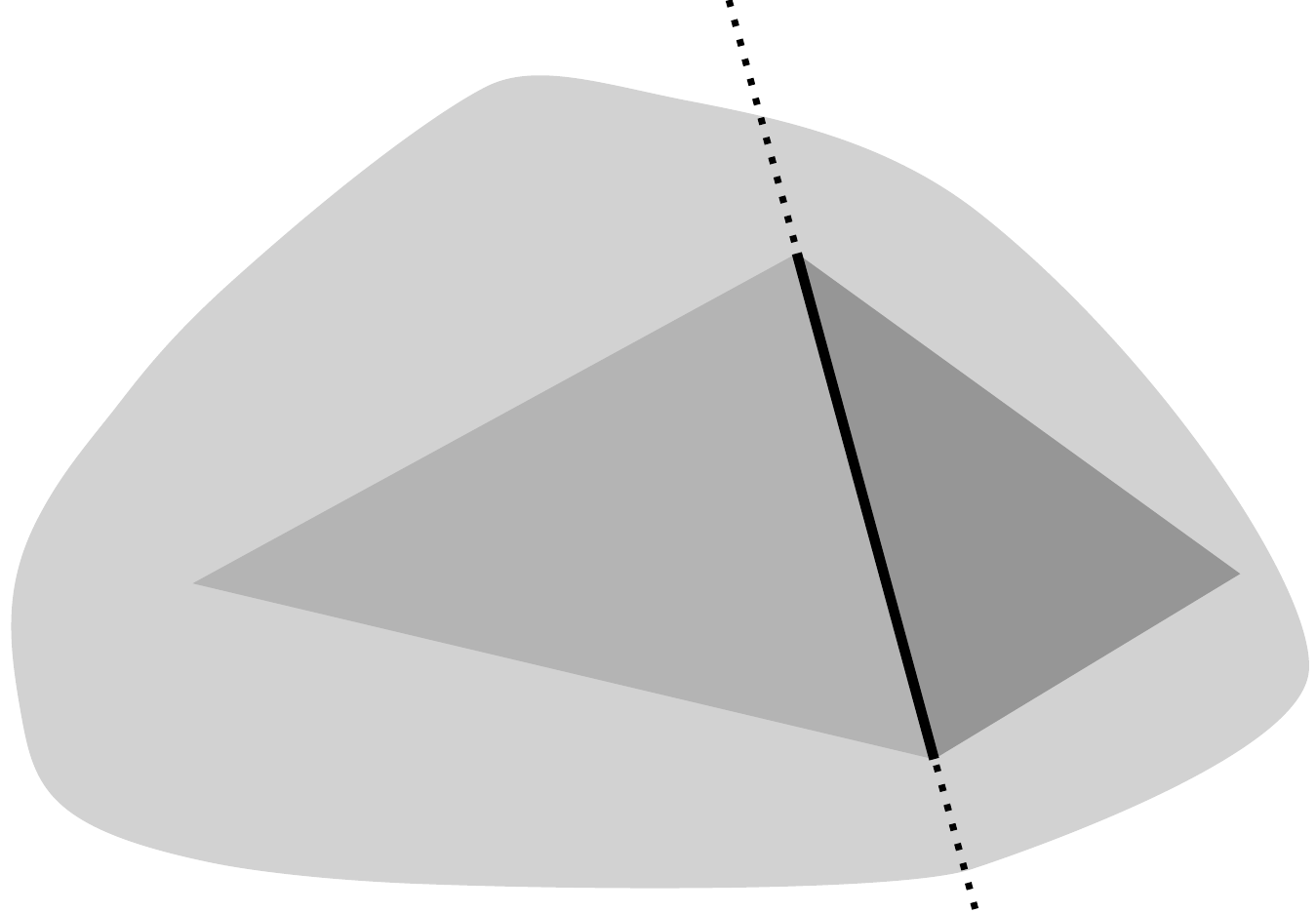}
     \caption{Corollary \ref{coro_flip}.}
    \label{figure_reflection}
\end{figure}
The last statement is because $U\cup I$ and its conjugate $\sigma(U)\cup I$ are glued along $I$ to form a conical singularity of $\ve{\psi}_\Omega$  at $q$, of cone angle bigger than $2\pi$ (see the first paragraph of this subsection).

\subsection{Hyperbolic triangle groups and projective deformations}\label{subsec_deformation}
We henceforth let $a$, $b$ and $c$ be positive integers satisfying  $\frac{1}{a}+\frac{1}{b}+\frac{1}{c}<1$ and fix an inclusion of the hyperbolic plane $\mathbb{H}^2$ into $\RP$ as an ellipse (Klein model), so that the isometry group of $\mathbb{H}^2$ identifies with the group $\SO(2,1)<\SL(3,\R)$ of projective transformations preserving the ellipse. For every projective line $L$ passing through $\mathbb{H}^2$, there is a unique reflection across $L$ belonging to $\SO(2,1)$, namely the hyperbolic orthogonal reflection of $\mathbb{H}^2$ across $L$.

A geodesic triangle $\Delta\subset\mathbb{H}^2$ with angles $\frac{\pi}{a}$, $\frac{\pi}{b}$ and $\frac{\pi}{c}$ is called an \emph{$(a,b,c)$-triangle}. The corresponding \emph{hyperbolic triangle group}, denoted by $\Gamma_\Delta$, is the subgroup of $\SO(2,1)$ generated by the reflections $\sigma_1,\sigma_2,\sigma_3\in\SO(2,1)$ across the three sides of $\Delta$. 
We let $\ttri_\Delta$ denote the connected component in the space of representations $\Hom(\Gamma_\Delta,\SL(3,\R))$ containing the inclusion $\Gamma_\Delta\hookrightarrow\SO(2,1)\subset\SL(3,\R)$. Then each $\rho\in\ttri_\Delta$, referred to as a \emph{projective deformation} of $\Gamma_\Delta$, also has a fundamental triangle $\Delta_\rho\subset\RP$ such that $\rho(\sigma_1)$, $\rho(\sigma_2)$ and $\rho(\sigma_3)$ are reflections across the sides of $\Delta_\rho$. A theorem of Tits (see \cite[Thm. 2.2]{marquis_coxeter}) implies that $\Delta_\rho$ is a fundamental domain for the action of $\rho(\Gamma_\Delta)$ on some properly convex domain $\Omega_\rho\subset\RP$, hence we get pictures like Figure \ref{figure_tits}. The theorem can be formulated in the current setting as:
\begin{proposition}\label{prop_titsset}
	Let $\Delta\subset\mathbb{H}^2$ be an $(a,b,c)$-triangle and $\rho\in\ttri_\Delta$ be a projective deformation of $\Gamma_\Delta$. Then the triangles $(\rho(\gamma).\Delta_\rho)_{\gamma\in\Gamma_\Delta}$ have disjoint interiors and the union
	$$
	\Omega_\rho:=\bigcup_{\gamma\in\Gamma}\rho(\gamma).\Delta_\rho
	$$
	is a properly convex domain in $\RP$ preserved by $\rho(\Gamma_\rho)$.
\end{proposition}
We call the topological quotient $\tri_\Delta:=\ttri_\Delta/\SL(3,\R)$ the \emph{moduli space} of projective deformations of $\Gamma_\Delta$. It has a distinguished point given by \emph{trivial deformations}, \ie representations $\rho\in\ttri_\Delta$ conjugate to the inclusion $\Gamma_\Delta\to\SL(3,\R)$, which are characterized by the property that $\Omega_\rho$ is an ellipse.
%It can be shown that $\rho$ is trivial if and only if the convex domain $\Omega_\rho$ from Proposition \ref{prop_titsset} is an ellipse. 
It turns out that if one of the integers $a$, $b$ and $c$ equals $3$, then this is the only element of $\tri_\Delta$; otherwise, $\tri_\Delta$ can be identified with the real line (see \cite{goldman_geometric, nie_tits}):
\begin{proposition}\label{prop_para1}
Let $\Delta\subset\mathbb{H}^2$ be an $(a,b,c)$-triangle. Then the space $\tri_\Delta$ is not a single point if and only if $a,b,c\geq3$. In this case, $\tri_\Delta$ is homeomorphic to $\R$.
\end{proposition}

%Moreover, when $\min\{a,b,c\}\geq3$, we can suppose $0\in\R\cong\tri$ correspondences to the identity representation $\rho_0$ and  choose a representative $\rho_t\in\ttri$ of each $t\in\R\cong\tri$ in such a way that every $\rho_t$ has the same fundamental triangle $\Delta\subset E\cong\mathbb{H}^2$ as $\rho_0$. Namely,  $\rho_t(\sigma_1)$, $\rho_t(\sigma_2)$ and $\rho_t(\sigma_3)$ are projective reflections whose axes are the lines passing through the three sides of $\Delta$. 

\subsection{Pick differential of $\Omega_\rho$}\label{subsec_at}
%We describe in this section the Pick differential of the properly convex domain $\Omega_\rho$ from Proposition \ref{prop_titsset}.
While for a trivial deformation $\rho\in\ttri_\Delta$ the domain $\Omega_\rho$ is an ellipse and hence $\ve{\psi}_{\Omega_{\rho}}$ is constantly zero (see \S \ref{subsec_affinesphere}), we give in this subsection a description of $\ve{\psi}_{\Omega_\rho}$ for nontrivial $\rho$. Given $t\in\R^*$, letting $\ve{A}_t$ denote the closed triangle in $(\mathbb{C},\dz^3)$ with vertices $0$, $t$ and $te^{\frac{\pi\ima}{3}}$, considered as a $\frac{1}{3}$-translation surface with real boundary in the sense of Definition \ref{def_realboundary} (see Figure \ref{figure_triangle}), we show:
\begin{proposition}\label{prop_trho} 
Let $\Delta\subset\mathbb{H}^2$ be an $(a,b,c)$-triangle with $a,b,c\geq3$. Then for each nontrivial projective deformation $\rho\in\ttri_\Delta$ of $\Gamma_\Delta$, there is $t=t(\rho)\in\R^*$ such that the fundamental triangle $\Delta_\rho\subset(\Omega_\rho, \ve{\psi}_{\Omega_\rho})$ is isomorphic, as a $\frac{1}{3}$-translation surface with boundary, to $\ve{A}_t$. Moreover, the map from $\ttri_\Delta$ to $\R$ sending $\rho$ to $t(\rho)$ and sending the trivial deformations to $0$ induces a homeomorphism from $\tri_\Delta=\ttri_\Delta/\SL(3,\R)$ to $\R$.
	\begin{figure}[h]
	\labellist
	\pinlabel {$|t|$} at 213 135
	\pinlabel {$t$} at 640 5
	\pinlabel {$0$} at 300 -10
	\endlabellist
	\centering
	\includegraphics[width=10cm]{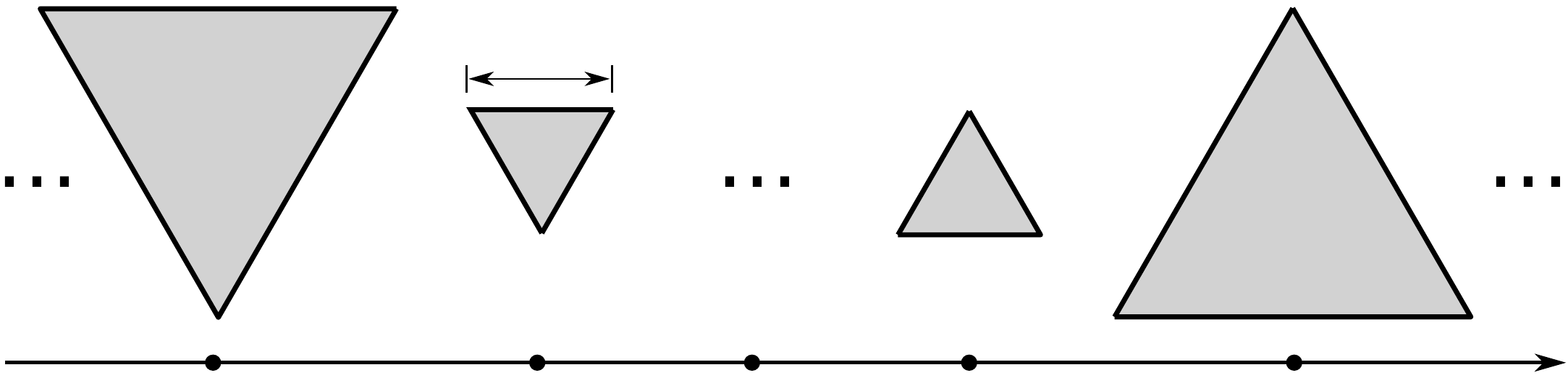}
	\caption{$\ve{A}_t$.}
	\label{figure_triangle}
\end{figure}
\end{proposition}
We postpone the proof to the end of this section, and proceed to explain how the proposition allows us to understand $\ve{\psi}_{\Omega_\rho}$ and prove Theorem \ref{thm_intro3}.

Let $\Gamma_\Delta^+$ denote the index $2$ subgroup of $\Gamma_\Delta$ formed by orientation-preserving automorphisms of $\mathbb{H}^2$, so that $\rho(\Gamma_\Delta^+)$ is also the subgroup of $\rho(\Gamma_\Delta)$ preserving the orientation of $\Omega$. If we identify $\Delta_\rho$ with $\ve{A}_t$ as in the proposition, then every $\gamma(\Delta_\rho)$ with $\gamma\in\Gamma_\Delta^+$, shown in Figure \ref{figure_tits} as a shaded triangle, also gets identified with $\ve{A}_t$, whereas by Corollary \ref{coro_flip}, every blank triangle $\gamma(\Delta_\rho)$ ($\gamma\in\Gamma_\Delta\setminus\Gamma_\Delta^+$) in the figure is identified with the conjugate of $\ve{A}_t$, which is exactly $\ve{A}_{-t}$. 

Therefore, the cubic differential $\ve{\psi}_{\Omega_\rho}$ can be understood as obtained by first endowing each shaded (resp. blank) triangle in Figure \ref{figure_tits} with the structure of $\ve{A}_t$ (resp. $\ve{A}_{-t}$), then gluing along the sides. In particular, the zeros of $\ve{\psi}_{\Omega_\rho}$, which are the conical singularities  of the resulting $\frac{1}{3}$-translation surface (see \S \ref{subsec_real}), can only be at the vertices of these triangles.
More precisely, if $p\in\Omega_\rho$ is a common vertex of $n+3$ shaded triangles ($n\geq0$), then $p$ is a conical singularity of cone angle $2\pi+\frac{2\pi}{3}n$ if $n\geq1$ and is an ordinary point if $n=0$. Noting that the number of shaded triangles sharing a vertex can only be either $a$, $b$ or $c$, we can deduce Theorem \ref{thm_intro3} from Theorem \ref{thm_intro2}:

\begin{proof}[Proof of Theorem \ref{thm_intro3}]
	We may suppose that the homeomorphism $\tri_\Delta\cong\R$ in the assumption of the corollary is the one from Proposition \ref{prop_trho}, so that the Pick differential of $\Omega_{\rho_t}$ is constructed from $\ve{A}_t$ and $\ve{A}_{-t}$ as just explained. It satisfies the assumption of Theorem \ref{thm_intro2}, and has a zero of multiplicity $n$ exactly when $n\in\{a-3, b-3, c-3\}\setminus\{0\}$. So the required conclusion follows from Theorem \ref{thm_intro2}.
\end{proof}
To prove Proposition \ref{prop_trho}, note that Corollary \ref{coro_flip} already implies that the fundamental triangle $\Delta_\rho$, as a sub-surface of the $\frac{1}{3}$-translation surface $(\Omega_\rho,\ve{\psi}_{\Omega_\rho})$, is a $\frac{1}{3}$-translation surfaces with real boundary. However, this does not immediately imply, for example, that the boundary $\pa\Delta_\rho$ only has three corners (see Definition \ref{def_realboundary}), or the corners can only be at the three vertices. Compare the remark below. 

\begin{proof}[Proof of Proposition \ref{prop_trho}]
To prove the first statement, it is sufficient to show that the $\frac{1}{3}$-translation surface $\Delta_\rho$ (with real boundary) does not have conical singularities in its interior, and that its corners are exactly the three vertices, each with angle $\frac{2\pi}{3}$.
To this end, we let $V\subset\pa\Delta_\rho$ denote the three vertices and apply the Gauss-Bonnet formula (Theorem \ref{thm_gaussbonnet}) to get
$$
-\frac{2\pi}{3}\sum_{p\in Z(\Delta_\rho)}n(\Delta_\rho,p)-\frac{\pi}{3}\sum_{q\in C(\Delta_\rho)\cap V}l(\Delta_\rho,q)-\frac{\pi}{3}\sum_{q\in C(\Delta_\rho)\setminus V}l(\Delta_\rho,q)=2\pi
$$
(see \S \ref{subsec_real} for the notations). By the discussions in \S \ref{subsec_real}, the terms on the left-hand side have the following constraints:
\begin{itemize}
	\item each $n(\Delta_\rho,p)$ is a positive integer;
	\item each $l(\Delta_\rho, q)$ is an integer no less than $-2$;
	\item if $q\in C(\Delta_\rho)\setminus V$ then $l(\Delta_\rho, q)$ is positive. 
\end{itemize}
Note that the last constraint means that the angle $\angle(\Delta_\rho,q)$ is bigger than $\pi$, see Corollary \ref{coro_flip}.
One easily checks that the only possible situation satisfying these constraints is given by $Z(\Delta_\rho)=\emptyset$, $C(\Delta_\rho)=V$, $l(\Delta_\rho,q)=-2$ (for all of the three $q\in V$). This establishes the first statement.

For the second statement, note that the map $\ttri_\Delta\to\R$, $\rho\mapsto t(\rho)$ (setting $t(\rho)=0$ if $\rho$ is a trivial deformation) is invariant under the $\SL(3,\R)$-action on $\ttri_\Delta$, and the induced map $\tri_\Delta\to\R$ is bijective by the construction of Pick differentials. Finally, to prove that the map is continuous, we can take a continuous path $(\rho_s)_{s\in\R}$ in $\ttri_\Delta$ such that every $\rho_s$ has the same fundamental triangle $\Delta_{\rho_s}=\Delta$, and only need to show that $t(\rho_s)$ depends continuously on $s$. To this end, let $\ve{\psi}_s:=\ve{\psi}_{\Omega_{\rho_s}}$ denote the Pick differential of $\Omega_{\rho_s}$, and  $I\subset\pa\Delta$ be a side of $\Delta$, endowed with the orientation induced from that of $\Delta$. The first statement proved above implies that
$$
t(\rho_s)=\int_I\ve{\psi}^\frac{1}{3}_s
$$ 
(the integration can be written more concretely as $\int_0^1\ve{\psi}_s(\dot{\gamma}(\tau),\dot{\gamma}(\tau),\dot{\gamma}(\tau))^\frac{1}{3}\dif \tau$, where $\gamma:[0,1]\to I$ is a parametrization of $I$; the cubic root is defined unambiguously because $I$ is a real trajectory). Therefore, the required continuity follows from the well-known fact that the convex domain $\Omega_\rho\in\C$ depends continuously on the representation $\rho$, and the fact that given $\Omega\in\C$, the restriction of $\ve{\psi}_\Omega$ to a compact set $K\subset\Omega$ depends continuously on $\Omega$ (see Lemma \ref{lemma_benoisthulin}; here we take $I$ as $K$).
\end{proof}
\begin{remark}
For a properly convex domain $\Omega$ generated by reflections across the sides of a convex polygon $D\subset\RP$ with more than three sides, one can extend the above argument to describe the $\frac{1}{3}$-translation structure on $D$, and hence describe the Pick differential $\ve{\psi}_\Omega$. In this case, a corner of $D$ (with interior angle bigger than $\pi$) can lie on a side, and the moduli space of the possible $\frac{1}{3}$-translation structures has dimension higher than $1$, which nonetheless can still be identified with the moduli spaces of projective deformations of $\Omega$. For instance, when $D$ is a quadrilateral, the first and third examples in Figure \ref{figure_real} can occur as the $\frac{1}{3}$-translation structure on $D$, as long as all the four corners with index $-2$ correspond to the vertices of $D$; whereas the second example cannot, because it has five corners with index $-2$, and such a corner cannot lie on a side.
\end{remark}

\bibliographystyle{amsalpha} 
\bibliography{degeneration}
\end{document}